\documentclass[a4paper, 12pt]{article}
\usepackage{}

\usepackage{mathrsfs}
\usepackage{epsfig}
\usepackage{amsmath}
\usepackage{amssymb}
\usepackage{latexsym}
\usepackage{amsfonts}
\usepackage{amsthm}

\usepackage{bbm}

\usepackage[numbers,sort&compress]{natbib}
\usepackage{graphicx}
\ifx\pdfoutput\undefined  \else
 \fi

\usepackage[centerlast]{caption2}

\usepackage{color}

\usepackage{txfonts}
\usepackage{amssymb}
\usepackage{mathrsfs}
\usepackage{amsfonts}

\newtheorem{theorem}{Theorem}[section]
\newtheorem{lemma}[theorem]{Lemma}
\newtheorem{cor}[theorem]{Corollary}
\newtheorem{prop}[theorem]{Proposition}
\newtheorem{conj}[theorem]{Conjecture}
\newtheorem{remark}[theorem]{Remark}
\newtheorem{definition}[theorem]{Definition}
\newtheorem{exam}[theorem]{Example}
\newtheorem{problem}[theorem]{Problem}
\newtheorem{question}[theorem]{Question}

\usepackage{latexsym,amssymb}
\usepackage{amsmath}

\begin{document}

\title{{The universal zero-sum invariant and weighted zero-sum for infinite abelian groups}}
\author{
Guoqing Wang\\
\small{School of
Mathematical Sciences, Tiangong University, Tianjin, 300387, P. R. China}\\
\small{Email: gqwang1979@aliyun.com}
\\
}
\date{}
\maketitle

\begin{abstract} Let $G$ be an abelian group, and let $\mathcal F (G)$ be the free
commutative monoid with basis $G$. For $\Omega \subset \mathcal F (G)$, define the universal zero-sum invariant ${\mathsf d}_{\Omega}(G)$ to be the smallest integer $\ell$ such that
every sequence $T$ over $G$ of length $\ell$ has a subsequence in $\Omega$. The invariant ${\mathsf d}_{\Omega}(G)$ unifies many classical zero-sum invariants. Let $\mathcal B (G)$ be the submonoid of $\mathcal F (G)$ consisting of
all zero-sum sequences over $G$, and let $\mathcal A (G)$ be the set consisting of all minimal zero-sum sequences over $G$. The empty sequence, which is the identity of  $\mathcal B (G)$, is denoted by $\varepsilon$. The well-known Davenport constant ${\rm D}(G)$ of the group $G$ can be also represented as ${\mathsf d}_{\mathcal B (G)\setminus \{\varepsilon\}}(G)$ or ${\mathsf d}_{\mathcal A (G)}(G)$ in terms of the universal zero-sum invariant.
Notice that $\mathcal A (G)$ is the unique minimal generating set of the monoid $\mathcal B (G)$ from the point of view of Algebra. Hence, it would be interesting to determine whether $\mathcal A (G)$ is minimal to represent the Davenport constant or not for a general finite abelian group $G$. In this paper, we show that except for a few special classes of groups, there always exists a proper subset $\Omega$ of $\mathcal A (G)$ such that  ${\mathsf d}_{\Omega}(G)={\rm D}(G)$. Furthermore, in the setting of finite cyclic groups, we discuss the distributions of all minimal sets by determining their intersections.

By connecting the universal zero-sum invariant with weights, we make a study of zero-sum problems in the setting of {\sl infinite} abelian groups.
The universal zero-sum invariant ${\mathsf d}_{\Omega; \Psi}(G)$ with weights set $\Psi$ of homomorphisms of groups is introduced for all abelian groups.
The weighted Davenport constant ${\rm D}_{\Psi}(G)$ (being an special form of the universal invariant with weights) is also
investigated for infinite abelian groups.
Among other results, we obtain the necessary and sufficient conditions such that ${\rm D}_{\Psi}(G)<\infty$ in terms of the weights set $\Psi$ when $|\Psi|$ is finite. In doing this, by using the Neumann Theorem on Cover Theory for groups we establish a connection between the existence of a finite cover of an abelian group $G$ by cosets of some given subgroups of $G$, and the finiteness of weighted Davenport constant.
\end{abstract}


\section {Introduction}

Let $G$ be a finite abelian group. The Davenport constant ${\rm D}(G)$ of $G$ is defined as
the smallest positive integer $\ell$ such that, every sequence of $\ell$ terms from $G$
contains some terms with sum being the identity element. This invariant was first formulated by K. Rogers \cite{Rogers}, and popularized by H. Davenport in the 1960$'$s, notably for its link with algebraic number theory (as
reported in \cite{Olson1}), and has been investigated extensively in the past 60 years.
This combinatorial
invariant was found with applications in other areas, including Factorization Theory of Algebra (see
\cite{CziszterDoGerolding, GRuzsa, GH}), Classical Number Theory, Graph Theory, and Coding Theory. For example, the Davenport
constant has been applied by W.R. Alford, A. Granville and C. Pomerance \cite{AGP} to prove that there are
infinitely many Carmichael numbers.

What is more important, a lot of researches were motivated by
the Davenport constant together with the celebrated EGZ Theorem obtained by P. Erd\H{o}s, A. Ginzburg and A. Ziv \cite{EGZ} in 1961 on additive properties of sequences in the setting of finite groups (mainly in finite abelian groups), which have been developed into
a branch, called zero-sum theory (see \cite{GaoGeroldingersurvey} for a survey), in Combinatorial Number Theory. Zero-sum in the setting of finite abelian groups, deals with the problems associated with the existence of a substructure with given additive properties, in a combinatorial structure such as set, multiset, hypergraph, etc. While, the conditions for the existence of substructures with distinct flavours of properties are formulated as distinct invariants, called zero-sum constants. Hence, a lot of zero-sum constants were investigated in zero-sum theory. To understand the common behaviors of all those zero-sum constants, a new way
to view all zero-sum constants was proposed \cite{GaoLiPengWang} which is stated in Definition \ref{definition domenga}. For more researches, one is refer to \cite{GaoHong1,GaoHong2}.

Let $\mathcal F (G)$ be the free
commutative monoid, multiplicatively written, with basis $G$, i.e., consisting of all finite sequences of terms from the group $G$. Let $\mathcal B (G)$ be the submonoid of $\mathcal F (G)$ consisting of
all {\sl zero-sum} sequences over $G$.

\begin{definition} \label{definition domenga} \cite{GaoLiPengWang}
For any nonempty set $\Omega \subset \mathcal F (G)$, we define the universal zero-sum invariant ${\mathsf d}_{\Omega}(G)$ to be the smallest positive integer $\ell$ (if it exists, otherwise ${\mathsf d}_{\Omega}(G)=\infty$) such that
every sequence over $G$ of length $\ell$ has a subsequence in $\Omega$.
\end{definition}

Observe that  ${\mathsf d}_{\Omega}(G)=1$ if the empty sequence $\varepsilon$ belongs to $\Omega$. Hence, we need only to consider the case that $\varepsilon\notin \Omega$ in what follows. The monoid $\mathcal B (G)$ is the common setting for both zero-sum theory and factorization theory of algebra  (see \cite{GH}), and moreover, in terms of ${\mathsf d}_{\Omega}(G)$,
the investigation on zero-sum problems are mainly associated with all different types of subsets $\Omega$ of $\mathcal B (G)$. That is, by taking distinct subset $\Omega$ of $\mathcal B (G)\setminus \{\varepsilon\}$, the universal invariant ${\mathsf d}_{\Omega}(G)$ reduces to the corresponding zero-sum constant.
For example, the Davenport constant of a finite abelian group $G$ can be denoted in terms of ${\mathsf d}_{\Omega}(G)$ as ${\rm D}(G)={\mathsf d}_{\mathcal B (G)\setminus \{\varepsilon\}}(G)$.
 Meanwhile, since every sequence $T\in \mathcal B (G)\setminus \{\varepsilon\}$ can be factorized into a product of elements of $\mathcal A (G)$ which is the subset of $\mathcal B (G)$ consisting of
all minimal zero-sum sequences over $G$ (see \cite{GH}), we have that $${\mathsf d}_{\mathcal A (G)}(G)={\mathsf d}_{\mathcal B (G)\setminus \{\varepsilon\}}(G)={\rm D}(G).$$
From the point of Factorization Theory in Algebra, the set $\mathcal A (G)$ consists of all irreducible elements of the monoid $\mathcal B (G)$ and it is definitely a minimal generating set of $\mathcal B (G)$.  However, from the point of invariants in Zero-sum Theory,  we know little about the minimality of $\mathcal A (G)$ with respect to the Davenport constant ${\rm D}(G)$.
Hence, the following basic question arouse on the minimality of $\mathcal A (G)$ to represent the Davenport constant:

\begin{question}\label{question}
Is $\Omega=\mathcal A (G)$ minimal with respect to ${\mathsf d}_{\Omega}(G)={\rm D}(G)$, i.e., does there exist $\Omega'\subsetneq\mathcal A (G)$ such that ${\mathsf d}_{\Omega'}(G)={\rm D}(G)$?
\end{question}

To formulate the questions more precisely, we shall need the following definition which was introduced  first in \cite{GaoLiPengWang}.

\begin{definition}\label{Definition minimal set}  (see \cite{GaoLiPengWang}, Section 4) Let $G$ be a finite abelian group, and let $t>0$ be an integer.  A set $\Omega\subset \mathcal B (G)$ is called {\bf minimal} with respect to $t$ provided that ${\mathsf d}_{\Omega}(G)=t$ and ${\mathsf d}_{\Omega'}(G)\neq t$ for any subset $\Omega'\subsetneq\Omega$.
In particular, if $G$ is finite and $t={\rm D}(G)$ then we just call $\Omega$ a minimal set for short.
\end{definition}

\begin{remark}\label{remark} By definition, we have that ${\mathsf d}_{\Omega_1}(G)\geq {\mathsf d}_{\Omega_2}(G)$ for any $\Omega_1\subset \Omega_2$.
\end{remark}

If $\mathcal A (G)$ is not a minimal set then the following question proposed in \cite{GaoLiPengWang} is interesting.

\begin{question} (see \cite{GaoLiPengWang}, Question 1) \label{question1}
What can be said about minimal sets $\Omega$?
\end{question}

Moreover, the investigation of minimal sets also closely relates to some classical zero-sum problems in cyclic groups. For example, the Lemke-Kleitman conjecture stated:

\begin{conj} \cite{LemkeKleitman} (Lemke-Kleitman conjecture)  \label{prop Conjecture}
Every sequence $T$ of length $n$ over the cyclic group $\mathbb{Z}_n$ contains a subsequence of Index one.
\end{conj}

Conjecture \ref{prop Conjecture} was disproved, in general,  by the author together with Gao, Li, Peng and Plyley (see \cite{Index}, Theorem 1.2). However, it remains open for the cyclic group $\mathbb{Z}_p$ with $p$ a prime. The above Lemke-Kleitman conjecture for prime $p$ is equivalent to the following related to the minimal sets.

\begin{conj}\label{prop equivalent proposition} There exists a
minimal set $\Omega\subset\mathcal A (\mathbb{Z}_p)$ with  $\Omega\subset \{V\in \mathcal A (\mathbb{Z}_p): V \mbox{ has Index } 1\}$.
\end{conj}

To obtain the general information on the distributions of minimal sets, W. Gao et al. \cite{GaoHong1} introduced the following essential definition.

\begin{definition}\label{Definition Qk(G)} \ For every positive integer $t\geq  {\rm D}(G)$, let
$$Q_t(G) = \bigcap\limits_{\Omega \subset \mathcal{B}(G), \ d_{\Omega}(G)=t} \Omega.$$
\end{definition}

In \cite{GaoHong1}, the authors obtained some nice properties on $Q_t(G)$. However, a little was known on the set $Q_t(G)$ itself. Motivated by the above questions and associated researches done in the past, we shall determine the set $Q_t(G)$ for all $t$ in finite cyclic groups $G$ in Section 3.
We shall answer Question \ref{question} by showing that, expect for a few special classes of groups, the set $\mathcal A (G)$ is not a minimal set with respect to the Davenport constant in Section 4.

In the remaining sections of this paper, the universal zero-sum invariant ${\mathsf d}_{\Omega; \Psi}(G)$ with the weights set $\Psi$ of homomorphisms is introduced and investigated for all abelian groups $G$
(not necessarily finite).  The  weighted Davenport constant, as the first weighted zero-sum invariant,  was introduced
by Adhikari, Chen et al. (see \cite{Adhikari1, Adhikari2}) for finite abelian groups, and the idea to use homomorphisms as weights was due to Yuan and Zeng \cite{ZengYuanDiscrete}.
Since the idea `weights' was introduced for zero-sum,  in the setting of {\sl finite} abelian groups, many papers were written concerning with the weighted Davenport constant and some other weighted analogies of zero-sum invariants (see \cite{Adhikari3, GrynkiewicMarchan, GrynkiewicIsrael, MarchanOrSaSc, YuanZengEruo} e.g), including weighted Erd\H{o}s-Ginzburg-Ziv constant (see
Chapter 16 in the monograph \cite{Grynkiewiczmono}), the Harborth constant.
Recently, L.E. Marchan, O. Ordaz, I. Santos and W.A. Schmid \cite{MaOrSaSc} established a link between weighted Davenport
constants and problems of linear codes. Moreover, arithmetic investigations of monoids of weighted zero-sum sequences
over a finite abelian group were initiated by Schmid et al. \cite{Boukheche}, and investigated systematically by Geroldinger, Halter-Koch and Zhong \cite{GHZ}.

In the past, almost all zero-sum problems associated with groups were investigated in the scope of {\sl finite} abelian groups,
since the corresponding zero-sum invariants turn to be infinity or  make no sense if they were applied to  infinite abelian groups. While, with the `weights' and the universal zero-sum invariant, we may see some possible interesting questions in the realm of infinite abelian groups in this paper.

Let $F$ and $G$ be abelian groups with the homomorphism group ${\rm Hom}(F, G)$, and let $\Psi$ be a nonempty subset of ${\rm Hom}(F, G)$. For any $\Omega\subset \mathcal F (G)$, we define the $\Psi$-weighted universal zero-sum invariant, denoted ${\mathsf d}_{\Omega; \Psi}(G)$, to be the smallest integer $\ell$ (if it exists, otherwise ${\mathsf d}_{\Omega; \Psi}(G)=\infty$) such that any sequence over $F$ of length $\ell$ has a subsequence which can be mapped into $\Omega$ under some homomorphisms of $\Psi$. With the weighted universal zero-sum invariant ${\mathsf d}_{\Omega; \Psi}(G)$, some zero-sum problems will be motivated in a more general setting---infinite abelian groups. In particular, if $\Omega=\mathcal B (G)\setminus \{\varepsilon\}$ then, ${\mathsf d}_{\Omega; \Psi}(G)$ reduces to the weighted Davenport constant ${\rm D}_{\Psi}(G)$.

In Section 5,  we make a preliminary study of ${\mathsf d}_{\Omega; \Psi}(G)$ by showing that given any  positive integer $t$ and any set
$A\subset \mathcal F (G)$ such that ${\mathsf d}_{A; \Psi}(G)=t$, there exists a minimal set $\Omega$ contained in $A$ such that ${\mathsf d}_{\Omega; \Psi}(G)=t$, which is stated as Theorem \ref{Lemma result by Zorn's Lemma}.
With the weighted Davenport constant ${\rm D}_{\Psi}(G)$ in the setting of infinite abelian groups, among other results, we give the necessary and sufficient conditions such that ${\rm D}_{\Psi}(G)<\infty$ in terms of the weights set $\Psi$ which can be seen in Theorem \ref{Prop finiteness}.  In the concluding Section 6, by using the celebrated Neumann Theorem on Cover Theory for groups, we establish a connection between the existence of a finite cover of an abelian group $G$ by cosets of some given subgroups of $G$, and the finiteness of weighted Davenport constant.

\section {Notation}

The notation and terminologies  on sequences are consistent with \cite{GaoGeroldingersurvey} and [\cite{GH}, Chapter 5]. Let ${\cal F}(G)$
be the free commutative monoid, multiplicatively written, with basis
$G$. By $T\in {\cal F}(G)$, we mean $T$ is a sequence of terms from $G$ which is
unordered, repetition of terms allowed. Denote $[x,y]=\{z\in \mathbb{Z}: x\leq z\leq y\}$ for integers $x,y\in \mathbb{Z}$. Say
$T=a_1a_2\cdot\ldots\cdot a_{\ell}$, where $a_i\in G$ for $i\in [1,\ell]$.
The sequence $T$ can be also denoted as  $\prod\limits_{a\in G}a^{{\rm v}_a(T)},$ where ${\rm v}_a(T)$ is a nonnegative integer and
means that the element $a$ occurs ${\rm v}_a(T)$ times in the sequence $T$. By $|T|$ we denote the length of the sequence, i.e., $|T|=\sum\limits_{a\in G}{\rm v}_a(T)=\ell.$ By $\varepsilon$ we denote the empty sequence in ${\cal F}(G)$ with $|\varepsilon|=0$.
By ${\rm supp}(T)$  we denote the set consisting of all distinct
elements which occur in $T$. We call $T'$
a subsequence of $T$ if ${\rm v}_a(T')\leq {\rm v}_a(T)$, for each element $a\in G,$ denoted by $T'\mid T,$ moreover, we write $T^{''}=T\cdot  T'^{-1}$ to mean the unique subsequence of $T$ with $T'\cdot T^{''}=T$.
Let $\sigma(T)=a_1+\cdots +a_{\ell}$ be the sum of all terms from $T$. In particular, we adopt the convention that $\sigma(\varepsilon)=0_G$.
Let $\Sigma(T)$ be the set consisting of the elements of $G$ that can be represented as a sum of one or more terms from $T$, i.e., $\Sigma(T)=\{\sigma(T'): T' \mbox{ is taken over all nonempty subsequences of }T\}$.
A sequence $T'$ is said to be equal to $T$, denoted $T'=T$, if ${\rm v}_a(T')={\rm v}_a(T)\ \ \mbox{for each element}\ \ a\in G.$ We call $T$ a {\bf zero-sum} sequence if $\sigma(T)=0$, and furthermore, call $T$ a {\bf minimal zero-sum} sequence, if $T$ is a nonempty zero-sum sequence and $T$ contains no nonempty proper ($\neq T$) zero-sum subsequence. We call $T$ a {\bf zero-sum free} sequence, if $T$ contains no nonempty zero-sum subsequence. Let $\mathcal B (G)$ be the submonoid of $\mathcal F (G)$ consisting of
all zero-sum sequences over $G$. Let $\mathcal A (G)$ be the subset of $\mathcal F (G)$ consisting of
all {\sl minimal} zero-sum sequences over $G$.

Let $F$ and $G$ be abelian groups, and let ${\rm Hom}(F, G)$ be the group consisting of all homomorphisms from $F$ to $G$. Let $\Psi$ be a nonempty subset of ${\rm Hom}(F, G)$.
Let $T\in \mathcal{F}(F)$.
Let $V=a_1\cdot\ldots\cdot a_k$ be a subsequence of $T$ ($k\geq 0$), and let $\psi_1,\ldots, \psi_{k}\in \Psi$ (not necessarily distinct). The sequence $\prod\limits_{i\in [1,k]}\psi_i(a_i)$ is called a $\Psi$-subsequence of $T$ (noticing that $\prod\limits_{i\in [1,k]}\psi_i(a_i)$ is a sequence over $G$). In particular, $\prod\limits_{i\in [1,k]}\psi_i(a_i)$ is called a {\bf $\Psi$-zero-sum} subsequence of $T$ if $\sum\limits_{i\in [1,k]}\psi_i(a_i)=0_G$. If $T$ has no nonempty $\Psi$-zero-sum subsequence, we say $T$ is  $\Psi$-zero-sum free.  To make the above definition more precise, we restate them as follows.
Denote
$$\Psi(V)=\{L\in \mathcal{F}(G): L=\psi_1(a_1)\cdot\ldots\cdot \psi_k(a_k) \mbox{ for some } \psi_1,\ldots,\psi_k\in \Psi\}.$$
Any sequence $W\in \bigcup\limits_{V\mid T} \Psi(V)$ is said to be a $\Psi$-subsequence of $T$, in particular, if $\sigma(W)=0_G$ then $W$ is called a $\Psi$-zero-sum subsequence of $T$. If $(\bigcup\limits_{V\mid T} \Psi(V))\cap \mathcal B (G)=\{\varepsilon\}$ then, the sequence $T$ is called $\Psi$-zero-sum free.

Note that for the case $F=G$ and $\Psi$ consisting of only the automorphism $1_G$, the terminology $\Psi$-zero-sum ($\Psi$-zero-sum free) is consistent to zero-sum (zero-sum free).

Let $G$ be a group, and let $\mathcal{S}=\{a_iG_i\}_{i=1}^k$ be a
system of left cosets of subgroups $G_1,\ldots,G_k$ (not necessarily distinct subgroups), where $k$ is called the {\sl size} of  $\mathcal{S}$. Provided that $G=\bigcup\limits_{i=1}^k (a_iG_i)$, we say the system $\mathcal{S}$ is a finite cover of $G$ , and in particular, we call $\mathcal{S}$ an {\sl irredundant} cover of $G$ if $\mathcal{S}$ has no proper subsystems covering $G$.

\section {Intersection of minimal sets in cyclic groups}

In this section, we shall
discuss the general distribution of minimal sets contained in $\mathcal{A}(G)$ in the setting of finite cyclic groups $G$ which can be seen in Theorem \ref{Theorem main result}. As an application of Theorem \ref{Theorem main result}, we could determine $Q_t(G)$ for all $t$ in Theorem \ref{Theorem intersection minimal}.

\begin{theorem}\label{Theorem main result} \ Let $G$ be a finite cyclic group of order $n\geq 3$, and let
$\Gamma=\{\Omega\subset \mathcal{A}(G): \Omega \mbox{ is a minimal set} \}$. Then
$$\bigcap\limits_{\Omega \in \Gamma} \Omega=
               \{a^{{\rm ord}(a)}: a\in G\}\cup \{g^{k}\cdot (-kg):  g\in G, {\rm ord}(g)=n, 1\leq k\leq n-2\}\cup X$$
where $X=\{g^2\cdot [(\frac{n}{2}-1)g]^2: g\in G, {\rm ord}(g)=n\}$ for $4\mid n$ and $n\geq 8$, and $X=\emptyset$ for otherwise.
\end{theorem}

It is well known that ${\rm D}(G)=n$ (see Lemma \ref{Lemma Davenport precise value}) for a cyclic group $G$ of order $n$. Since $Q_t(G)$ defines only for $t\geq {\rm D}(G)$, we give the following result on $Q_t(G)$ for all $t\geq n$.

\begin{theorem}\label{Theorem intersection minimal} \ Let $G$ be a finite cyclic group of order $n\geq 3$. Then

(i) $Q_{n}(G)=\{a^{{\rm ord}(a)}: a\in G, {\rm ord}(a)\in \{1, \frac{n}{2}, n\}\}\cup \{g^{k}\cdot (-kg):  g\in G, {\rm ord}(g)=n, 1\leq k\leq n-2\}$;

(ii) $Q_t(G)=\{a^{{\rm ord}(a)}: a\in G, {\rm ord}(a)=n\}$ for each $t\in [n+1,2n-1]$;

(iii) $Q_t(G)=\emptyset$ for any $t\geq 2n$.
\end{theorem}

By Theorem \ref{Theorem main result}, we can also derive the following.

\begin{cor}\label{Cor A(G) is not minimal} \ Let $G$ be a finite cyclic group of order $n\geq 3$. Then $\mathcal{A}(G)$ is minimal with respect to ${\rm D}(G)$ if and only if $n\leq 5$.
\end{cor}

\begin{proof} For $n\leq 5$, we can show the conclusion by checking that the set $\bigcap\limits_{\Omega \in \Gamma} \Omega$ given in Theorem \ref{Theorem main result} is equal to $\mathcal{A}(G)$. Suppose $n\geq 6$. Take an generator $g$ of $G$, and take the minimal sequence $V=g\cdot (2g) \cdot ((n-3)g)$. It is easy to see that $V\notin \bigcap\limits_{\Omega \in \Gamma} \Omega$, which implies that $\mathcal{A}(G)$ is not a minimal set.
\end{proof}

To prove Theorems \ref{Theorem main result} and \ref{Theorem intersection minimal}, some lemmas are necessary.

\begin{lemma}\label{Lemma existence for finite case} Let $G$ be a finite abelian group, and let $A\subset \mathcal{B}(G)$ such that $d_{A}(G)=t<\infty$. Then there exists a minimal set $\Omega\subset A$ with respect to $t$.
\end{lemma}

\begin{proof} Let $A'=\{V\in A: |V|\leq t\}$. Since $A'\subset A$, it follows from Remark \ref{remark} that $d_{A'}(G)\geq d_{A}(G)=t$. On the other hand, let $T\in \mathcal{F}(G)$ be an arbitrary sequence of length $t$. Since $d_{A}(G)=t$, it follows that $T$ contains a subsequence $W$, with $W\in A$. Since $|W|\leq |T|=t$, it follows that $W\in A'$. By the arbitrariness of $T$, we conclude that $d_{A'}(G)\leq t$, and so $$d_{A'}(G)=t.$$
Notice that $A'$ is finite. Then we can find a minimal set $\Omega$ with respect to $t$ which is contained in $A'$, and definitely contained in $A$, completing the proof.
\end{proof}

\begin{lemma}\label{lemma criterion}  Let $G$ be a finite abelian group, and let $V$ be a nonempty zero-sum sequence over $G$. Then we have the following conclusions.

(i) $V$ belongs to every minimal set contained in $\mathcal{A}(G)$  if and only if there exists a sequence $T\in \mathcal{F}(G)$ of length ${\rm D}(G)$ such that every minimal zero-sum sequence of $T$ is equal to $V$;

(ii) For any $t\geq {\rm D}(G)$, then $V\in Q_t(G)$ if and only if there exists a sequence $T\in \mathcal{F}(G)$ of length $t$ such that every nonempty zero-sum sequence of $T$ is equal to $V$.
\end{lemma}

\begin{proof} (i) To prove the necessity, we assume to the contrary, that $V$ belongs to every minimal set contained in $\mathcal{A}(G)$, but every sequence $T$ of length ${\rm D}(G)$  must contain a minimal zero-sum subsequence which is not equal to $V$. Let $A= \mathcal{A}(G)\setminus \{V\}$. By the hypothesis, we have that $d_{A}(G)\leq {\rm D}(G)$. By Remark \ref{remark}, we have that $d_{A}(G)\geq d_{\mathcal{A}(G)}(G)={\rm D}(G)$, and so $d_{A}(G)={\rm D}(G)$. By Lemma \ref{Lemma existence for finite case}, we have that $A$
contains a minimal set, but $V\notin A$, we derive a contradiction.

The sufficiency follows from the definition.

(ii) is just Lemma 5.7 in \cite{GaoHong1}.
\end{proof}

\begin{lemma}\label{lemma 1-smooth Integer sequences}  \ Let $T=\prod\limits_{i\in [1,\ell]} h_i$ be a sequence of positive integers of length $\ell$ such that  $\sum\limits_{i=1}^{\ell} h_i\leq 2\ell -3$.  Suppose that there exists $t\in [1,\ell]$ such that $h_t>1$. Then for each $x\in [h_t,\sum\limits_{i\in [1,\ell]\setminus \{t\}} h_i]$, we can find two subsequences $V_1,V_2$ of $T$ with distinct lengths such that $\sigma(V_1)=\sigma(V_2)=x$.
\end{lemma}

\begin{proof} We may assume that $h_i=1$, for all $i\in [1,s]$, and that
\begin{equation}\label{equation icreasing order of si}
2\leq h_{s+1}\leq h_{s+2}\leq \cdots \leq h_{\ell} \ \mbox{ where } s\in [0,\ell-1].
\end{equation}
Since $\sum\limits_{i=1}^{\ell} h_i\leq 2\ell -3$, it follows from the straightforward calculations that either (i) $s+1=\ell$ and $s\geq h_{\ell}+1$; or (ii) $s+1<\ell$ and $s\geq h_{\ell-1}+h_{\ell}-1$. To prove the conclusion, it suffices to assume  $t=s+1$ and  $$x\in [h_{s+1}, s+\sum\limits_{i\in [s+2,\ell]} h_i].$$
If (i) holds, then $1^x$ and $1^{x-h_{s+1}} h_{s+1}$
 are the desired subsequences of $T$ with distinct lengths, done. Hence, it remains to consider the case of (ii) holding. Let $m$ be the largest integer of $[s+1,\ell]$ such that
$\sum\limits_{i=s+1}^{m} h_i\leq x.$

Suppose $m<\ell$. By \eqref{equation icreasing order of si} and the maximality of $m$, we have
$x-\sum\limits_{i=s+1}^{m-1}h_i\leq (-1+\sum\limits_{i=s+1}^{m+1}h_i)-\sum\limits_{i=s+1}^{m-1}h_i=-1+h_m+h_{m+1}\leq -1+h_{\ell-1}+h_{\ell}\leq s$ (perhaps $m$ is equal to $s+1$ and then we admit $\sum\limits_{i=s+1}^{m-1}h_i=0$). Then
$1^{x-\sum\limits_{i=s+1}^{m-1} h_i} (\prod\limits_{i\in [s+1,m-1]} h_i)$ and $1^{x-\sum\limits_{i=s+1}^{m} h_i} (\prod\limits_{i\in [s+1,m]} h_i)$
 are the desired subsequences of $T$.

Suppose $m=\ell$.  Then $x-\sum\limits_{i=s+2}^{\ell} h_i\leq \sum\limits_{i\in [1,\ell]\setminus \{s+1\}}h_i-\sum\limits_{i=s+2}^{\ell}h_i=\sum\limits_{i=1}^{s}h_i
=s$. Then $1^{x-\sum\limits_{i=s+2}^{\ell} h_i} (\prod\limits_{i\in [s+2,\ell]} h_i)$ and $1^{x-\sum\limits_{i=s+1}^{\ell} h_i} (\prod\limits_{i\in [s+1,\ell]} h_i)$
 are the desired subsequences of $T$, done.
\end{proof}

 \begin{lemma} \label{Lemma folklore}  \   Let $T=\prod\limits_{i\in [1,\ell]} h_i$ be a sequence of positive integers such that  $\sum\limits_{i\in [1,\ell]} h_i>{\rm v}_1(T)$ and
 $\sum(T)=[1,\sum\limits_{i\in [1,\ell]} h_i]$. Then there exists $t\in [1,\ell]$ such that $1<h_t\leq {\rm v}_1(T)+1$.
 \end{lemma}

\begin{proof} Since ${\rm v}_1(T)+1\in \sum(T)$, we can represent ${\rm v}_1(T)+1$ as a sum of some terms from $T$. Then  we must have such $h_t$ as required.
\end{proof}

\noindent $\bullet$ Let $G$ be a cyclic group of order $n$, and let $g$ be a generator of $G$. For any $a\in G$, by ${\rm Ind}_g(a)$ we denote the unique positive integer $s\in [1,n]$  with $s g=a$. One fact will be used frequently is: if a nonempty sequence $V\in \mathcal{F}(G)$ satisfies that $\sum\limits_{a\mid V}{\rm Ind}_g(a)=n$, then $V$ is a minimal zero-sum sequence.

 \begin{lemma} \label{SachenChen} ([\cite{GRuzsa}, Theorem 5.1.8], \cite{Savchev-Chen}, \cite{Yuan}) \  Let $G$ be a cyclic group of order $n\geq 3$. If $T\in \mathcal{F}(G)$ is zero-sum free with $|T|>\frac{n}{2}$, then there exists some $g\in G$ with ${\rm ord}(g)=n$ such that

(i) $\sum\limits_{a\mid T} {\rm Ind}_g(a)<n$;

(ii) $\sum(\prod\limits_{a\mid T}{\rm Ind}_g(a))=[1, \sum\limits_{a\mid T} {\rm Ind}_g(a)]$ and in particular, $\sum(T)=\{ig: 1\leq i\leq \sum\limits_{a\mid T} {\rm Ind}_g(a)\}$;

(iii) ${\rm v_g}(T)\geq {\rm Ind}_g(a)$ for any $a\mid T$;
\end{lemma}

 \begin{lemma} \label{Lemma Qk} \cite{GaoHong1} Let $G$ be a finite abelian group with $|G|\geq 3$. For any positive integer $t\geq {\rm D}(G)$, we have $Q_{t+1}(G)\subset Q_t(G)$.
 \end{lemma}

\begin{lemma}\label{Lemma Graham conjecture} \cite{GaoHamiWang}
Let $G$ be a cyclic group with $|G|\geq 3$, and let $T\in \mathcal{F}(G)$ with $|T|\geq |G|$. If $|{\rm supp}(T)|\geq  3$, then $T$ has two nonempty zero-sum subsequences with distinct lengths.
\end{lemma}

Now we are in a position to prove Theorems \ref{Theorem main result} and \ref{Theorem intersection minimal}.

\noindent {\bf Proof of Theorem \ref{Theorem main result}}. \ For a sequence $W\in\{a^{{\rm ord}(a)}: a\in G\}$, or $W\in \{g^{k}\cdot (-kg):  {\rm ord}(g)=n, 1\leq k\leq n-2\}$, or $W\in  X$ (when $4\mid n$ and $n\geq 8$), we construct a sequence $T_W=a^{n}$, or  $T_W=g^{n-1}(-kg)$, or $T_W=g^{\frac{n}{2}}\cdot [(\frac{n}{2}-1)g]^{\frac{n}{2}}$ respectively, with length $|T_W|={\rm D}(G)=n$. We can verify that every minimal zero-sum subsequence of $T_W$ is equal to $W$.  It follows from Lemma \ref{lemma criterion} (i) that $\{a^{{\rm ord}(a)}: a\in G\}\cup \{g^{k}\cdot (-kg):  {\rm ord}(g)=n, 1\leq k\leq n-2\}\cup X\subset \bigcap\limits_{\Omega\in \Gamma} \Omega$.
To show the equality, we assume to the contrary that there exists some sequence
\begin{equation}\label{equation V in omegainGamma}
V\in \bigcap\limits_{\Omega\in \Gamma} \Omega
\end{equation} but
\begin{equation}\label{equation V not in}
V\notin \{a^{{\rm ord}(a)}: a\in G\}\cup \{g^{k}\cdot (-kg):  {\rm ord}(g)=n, 1\leq k\leq n-2\}\cup X.
\end{equation} Then
\begin{equation}\label{equation |supp(T)|geq 2}
|{\rm supp}(V)|\geq 2.
\end{equation}
By \eqref{equation V in omegainGamma} and Lemma \ref{lemma criterion} (i), there exists a sequence $T$ over $G$ of length $n$ such that every minimal zero-sum subsequence of $T$ is equal to $V$.
Fix an element $c\in {\rm supp}(V)$ with
\begin{equation}\label{equation vc(T)isminimal}
{\rm v}_c(T)=\min\limits_{x\mid V}\{{\rm v}_x(T)\}.
\end{equation}
By \eqref{equation |supp(T)|geq 2} and \eqref{equation vc(T)isminimal}, we have that
\begin{equation}\label{equation vc(T)leqn/2}
{\rm v}_c(T)\leq \frac{n}{2}
 \end{equation}
 and moreover, combining with the fact $Tc^{-{\rm v}_c(T)}$ is  zero-sum free, we have that
\begin{equation}\label{equation if and onlyif}
{\rm v}_c(T)=\frac{n}{2} \Leftrightarrow Tc^{-{\rm v}_c(T)}=b^{\frac{n}{2}} \mbox{ for some }b\in G \mbox{ with }{\rm ord}(b)=n.
\end{equation}

\noindent {\bf Claim A.}   There exists some generator $g$ of $G$ such that $\sum (\prod\limits_{a\mid Tc^{-{\rm v}_c(T)}}{\rm Ind}_g(a))=[1,\sum\limits_{a\mid Tc^{-{\rm v}_c(T)}} {\rm Ind}_g(a)]$ with $n-1\geq \sum\limits_{a\mid Tc^{-{\rm v}_c(T)}} {\rm Ind}_g(a)\geq |Tc^{-{\rm v}_c(T)}|\geq \frac{n}{2}$, and moreover, equality $\sum\limits_{a\mid Tc^{-{\rm v}_c(T)}} {\rm Ind}_g(a)=|Tc^{-{\rm v}_c(T)}|$ holds if and only if  $T=g^{n-{\rm v}_c(T)}c^{{\rm v}_c(T)}$.

\noindent {\sl Proof of Claim A.} \ If ${\rm v}_c(T)=\frac{n}{2}$, the conclusion follows from \eqref{equation if and onlyif} readily. By \eqref{equation vc(T)leqn/2}, it remains to consider the case of ${\rm v}_c(T)<\frac{n}{2}$, i.e., $|Tc^{-{\rm v}_c(T)}|>\frac{n}{2}$. Since $Tc^{-{\rm v}_c(T)}$ is  zero-sum free, it follows from
Lemma \ref{SachenChen} (i) and (ii) that there exists some generator $g$ of $G$ such that $\sum (\prod\limits_{a\mid Tc^{-{\rm v}_c(T)}}{\rm Ind}_g(a))=[1,\sum\limits_{a\mid Tc^{-{\rm v}_c(T)}} {\rm Ind}_g(a)]$ with $n-1\geq \sum\limits_{a\mid Tc^{-{\rm v}_c(T)}} {\rm Ind}_g(a)\geq  |Tc^{-{\rm v}_c(T)}|>\frac{n}{2}$. Moreover, if $\sum\limits_{a\mid Tc^{-{\rm v}_c(T)}} {\rm Ind}_g(a)=|Tc^{-{\rm v}_c(T)}|$, then ${\rm Ind}_g(a)=1$ for each $a\mid Tc^{-{\rm v}_c(T)}$, i.e.,  $Tc^{-{\rm v}_c(T)}=g^{|Tc^{-{\rm v}_c(T)}|}$, completing the proof of Claim A.
\qed

\noindent $\bullet$ In the remaining of the proof of Theorem \ref{Theorem main result}, we shall fix an generator $g$ (noting that $g\neq c$ since $1\in \sum (\prod\limits_{a\mid Tc^{-{\rm v}_c(T)}}{\rm Ind}_g(a))$) of $G$ such that the properties given in Claim A hold. Then we simply denote ${\rm Ind}_g(\cdot)$ by  ${\rm Ind}(\cdot)$.

Note that $$({\rm v}_c(T)-1){\rm Ind}(c)+\sum\limits_{a\mid Tc^{-{\rm v}_c(T)}} {\rm Ind}(a)\geq  n.$$ Otherwise, $n-1\geq ({\rm v}_c(T)-1){\rm Ind}(c)+\sum\limits_{a\mid Tc^{-{\rm v}_c(T)}} {\rm Ind}(a)\geq ({\rm v}_c(T)-1)2+|Tc^{-{\rm v}_c(T)}|=|T|+({\rm v}_c(T)-2)=n+({\rm v}_c(T)-2)$ implies that ${\rm v}_c(T)=1$ and $\sum\limits_{a\mid Tc^{-{\rm v}_c(T)}} {\rm Ind}(a)=|Tc^{-{\rm v}_c(T)}|$, and by Claim A, we derive $T=g^{n-1}c$.  Then we find a minimal zero-sum subsequence $V=g^{n-{\rm Ind}(c)}c$ of $T$, a contradiction with \eqref{equation V not in}.

Let \begin{equation}\label{equation mleq vc(T)-1}
m\in [1, {\rm v}_c(T)-1]
\end{equation}
be the least positive integer such that
\begin{equation}\label{equation mind(C)+ageq n}
m {\rm Ind}(c)+\sum\limits_{a\mid Tc^{-{\rm v}_c(T)}} {\rm Ind}(a)\geq  n.
\end{equation}
By applying Claim A and the minimality of $m$, we have that
\begin{equation}\label{equation (m-1)indleq n/2-1}
(m-1){\rm Ind}(c)\leq n-\sum\limits_{a\mid Tc^{-{\rm v}_c(T)}}{\rm Ind}(a)-1\leq \frac{n}{2}-1.
\end{equation}
Note that \begin{equation}\label{equation m Ind(c)leq n}
m {\rm Ind}(c)\leq n.
\end{equation}
This is because that if $m=1$ then ${\rm Ind}(c)\leq n$ follows from the definition of the notation ${\rm Ind}(\cdot)$, and if $m>1$ then $m {\rm Ind}(c)\leq 2(m-1){\rm Ind}(c)< n$ follows from \eqref{equation (m-1)indleq n/2-1}.

Since $\sum (\prod\limits_{a\mid Tc^{-{\rm v}_c(T)}}{\rm Ind}_g(a))=[1,\sum\limits_{a\mid Tc^{-{\rm v}_c(T)}} {\rm Ind}_g(a)]$, it follows from \eqref{equation mind(C)+ageq n} and \eqref{equation m Ind(c)leq n} that there exists a subsequence $L$ of $Tc^{-{\rm v}_c(T)}$ such that $\sum\limits_{a\mid  L} {\rm Ind}(a)=n-m {\rm Ind}(c)$ and $\sum\limits_{a\mid c^m \cdot L} {\rm Ind}(a)=n$, which implies that $c^m \cdot L$ is a minimal zero-sum subsequence of $T$
and so
$V=c^m \cdot  L$ with
\begin{equation}\label{equation vc(V)=m}
{\rm v}_c(V)=m.
\end{equation}

On the other hand, we have
\begin{equation}\label{equation (m+1)indgeq n+1}
(m+1){\rm Ind}(c)\geq n+1,
\end{equation}
otherwise $(m+1){\rm Ind}(c)\leq n$, it follows from \eqref{equation mind(C)+ageq n} and Claim A that there exist two subsequences $W_1,W_2$ (perhaps empty sequence) of $Tc^{-{\rm v}_c(T)}$ with $\sum\limits_{a\mid W_1} {\rm Ind}(a)=n-(m+1) {\rm Ind}(c)$ and $\sum\limits_{a\mid W_2} {\rm Ind}(a)=n-m {\rm Ind}(c)$,
combining with \eqref{equation mleq vc(T)-1}, then $W_1 c^{m+1}$ and $W_2 c^{m}$ are two minimal zero-sum subsequences of $T$ with $W_1 c^{m+1}\neq W_2 c^{m}$, a contradiction.

Combining \eqref{equation (m-1)indleq n/2-1} and \eqref{equation (m+1)indgeq n+1}, we calculate that  $$m<3.$$
Suppose that $m=1$. By \eqref{equation (m+1)indgeq n+1},  we drive that
\begin{equation}\label{equation n-Ind(c)leq n-1/2}
n-{\rm Ind}(c)\leq  \frac{n-1}{2}.
\end{equation}
Now we have
\begin{equation}\label{equation n-Ind(c)>vg(Tc-vc(T)}
n-{\rm Ind}(c)>{\rm v}_g(Tc^{-{\rm v}_c(T)}),
\end{equation}
 since otherwise $n-{\rm Ind}(c)\leq {\rm v}_g(Tc^{-{\rm v}_c(T)})$ implies that $V=g^{n-{\rm Ind}(c)} \cdot c$ is a minimal zero-sum subsequence of $T$, a contradiction with \eqref{equation V not in}.

 By Claim A, \eqref{equation n-Ind(c)leq n-1/2} and \eqref{equation n-Ind(c)>vg(Tc-vc(T)}, we can see that $\sum (\prod\limits_{a\mid Tc^{-{\rm v}_c(T)}}{\rm Ind}(a))=[1,\sum\limits_{a\mid Tc^{-{\rm v}_c(T)}} {\rm Ind}(a)]$ and ${\rm v}_1(\prod\limits_{a\mid Tc^{-{\rm v}_c(T)}}{\rm Ind}(a))={\rm v}_g(Tc^{-{\rm v}_c(T)})<\frac{n}{2}\leq \sum\limits_{a\mid Tc^{-{\rm v}_c(T)}} {\rm Ind}(a)$, combining with Lemma \ref{Lemma folklore}, we can find some  $x\mid Tc^{-{\rm v}_c(T)}$  such that $1<{\rm Ind}(x)\leq {\rm v}_g(Tc^{-{\rm v}_c(T)})+1$, and thus,
 $${\rm Ind}(x)\leq n-{\rm Ind}(c)\leq \left\lfloor\frac{n-1}{2}\right\rfloor\leq |Tc^{-{\rm v}_c(T)}|-1\leq \sum\limits_{a\mid Tc^{-{\rm v}_c(T)}x^{-1}} {\rm Ind}(a)=-{\rm Ind}(x)+\sum\limits_{a\mid Tc^{-{\rm v}_c(T)}} {\rm Ind}(a).$$
Then by applying Lemma \ref{lemma 1-smooth Integer sequences} to the integers sequence $\prod\limits_{a\mid Tc^{-{\rm v}_c(T)}} {\rm Ind}(a)$,
 we have that
 \begin{equation}\label{equation result by}
 \sum \limits_{a\mid Tc^{-{\rm v}_c(T)}} {\rm Ind}(a)\geq 2|\prod\limits_{a\mid Tc^{-{\rm v}_c(T)}} {\rm Ind}(a)|-2=2|Tc^{-{\rm v}_c(T)}|-2,
 \end{equation}
  since otherwise, we can find two subsequences $W_1,W_2$ of $Tc^{-{\rm v}_c(T)}$ with $W_1\neq W_2$ such that $\sum \limits_{a\mid W_1} {\rm Ind}(a)=\sum \limits_{a\mid W_2} {\rm Ind}(a)=n-{\rm Ind}(c)$, and thus, $W_1c$ and $W_2c$ are minimal zero-sum subsequences of $T$ with $W_1c\neq W_2c$, a contradiction.

 Note that
  \begin{equation}\label{equation vc(T)leq (n-1)/2}
 {\rm v}_c(T)\leq \frac{n-1}{2},
  \end{equation}
  since otherwise, by \eqref{equation vc(T)leqn/2} and \eqref{equation if and onlyif}, we have ${\rm v}_c(T)=\frac{n}{2}$ and $Tc^{-{\rm v}_c(T)}=g^{\frac{n}{2}}$ (noticing that $g\mid Tc^{-{\rm v}_c(T)}$ since $1\in \sum (\prod\limits_{a\mid Tc^{-{\rm v}_c(T)}}{\rm Ind}(a))$ by Claim A), combining with \eqref{equation n-Ind(c)leq n-1/2},  we have $g^{n-{\rm Ind}(c)}c$ is a minimal zero-sum subsequence of $T$,  a contradiction with \eqref{equation V not in}. This proves \eqref{equation vc(T)leq (n-1)/2}.

By \eqref{equation result by}, \eqref{equation vc(T)leq (n-1)/2} and Claim A, we have that
$n-1\geq \sum \limits_{a\mid Tc^{-{\rm v}_c(T)}} {\rm Ind}(a)\geq 2|Tc^{-{\rm v}_c(T)}|-2= 2(n-{\rm v}_c(T))-2\geq 2 (n-\frac{n-1}{2})-2= n-1$, which implies ${\rm v}_c(T)=\frac{n-1}{2}$. Combined with \eqref{equation |supp(T)|geq 2}, \eqref{equation vc(T)isminimal}, \eqref{equation n-Ind(c)leq n-1/2} and \eqref{equation n-Ind(c)>vg(Tc-vc(T)}, we conclude that there exists some $h\mid Tc^{-{\rm v}_c(T)}$ with $h\neq g$ such that ${\rm v}_h(T)\geq  {\rm v}_c(T)=\frac{n-1}{2}$.
Since ${\rm Ind}(h)\geq  2$, it follows that $\sum \limits_{a\mid Tc^{-{\rm v}_c(T)}}{\rm Ind}(a)={\rm Ind}(h) \cdot {\rm v}_h(T)+\sum \limits_{a\mid Tc^{-{\rm v}_c(T)}h^{-{\rm v}_h(T)}}{\rm Ind}(a)\geq 2 {\rm v}_h(T)+|Tc^{-{\rm v}_c(T)}h^{-{\rm v}_h(T)}|=2 {\rm v}_h(T)+|T|-{\rm v}_c(T)-{\rm v}_h(T)\geq |T|=n$,
a contradiction.  Hence,  $$m=2.$$

\noindent \textbf{Claim B.} For any $i\in [1,{\rm v}_c(T)]$ such that $i\equiv 1\pmod 2$, then ${\rm Ind}(ic) \leq {\rm v}_c(T)-1.$

\noindent {\sl Proof of Claim B.} Assume to the contrary that there exists some $j\in [1,{\rm v}_c(T)]$ such that
\begin{equation}\label{equation h is odd}
j\equiv 1\pmod 2
\end{equation}
and
\begin{equation}\label{equation Ind(hc)geqvcT}
{\rm Ind}(jc)\geq {\rm v}_c(T).
\end{equation}
 Since ${\rm v}_c(T)=n-|Tc^{-{\rm v}_c(T)}|\geq n-\sum\limits_{a\mid Tc^{-{\rm v}_c(T)}} {\rm Ind}(a)$, it follows from \eqref{equation Ind(hc)geqvcT} that $n-{\rm Ind}(jc)\leq \sum\limits_{a\mid Tc^{-{\rm v}_c(T)}} {\rm Ind}(a)$. By Claim A, there exists a subsequence $W$ of $Tc^{-{\rm v}_c(T)}$ with $\sum\limits_{a\mid W} {\rm Ind}(a)=n- {\rm Ind}(jc)$.
Then $W c^{j}$ is a zero-sum subsequence of $T$ with ${\rm v}_c(W c^{j})=j$. By \eqref{equation h is odd}, we can find a {\bf minimal} zero-sum subsequence of $W c^{j}$ with
$c$ appearing odd times, which definitely is not equal to $V$ by \eqref{equation vc(V)=m}, a contradiction. This proves Claim B. \qed

By \eqref{equation (m-1)indleq n/2-1} and \eqref{equation (m+1)indgeq n+1}, we have
\begin{equation}\label{equation 2n+2/3leqind(c)leq}
\frac{2n+2}{3}\leq  2 {\rm Ind}(c)\leq n-2,
 \end{equation}
 which implies that
\begin{equation}\label{equation ngeq 8}
n\geq 8
\end{equation}
 and
 $${\rm Ind}(c)\geq  3.$$
Recall that ${\rm v}_c(T)\leq \frac{n}{2}$ by \eqref{equation vc(T)leqn/2}.
Since $\frac{2n+2}{3}>\frac{n}{2}$ and ${\rm Ind}((i+2)c)$ is equal to the positive residue of ${\rm Ind}(ic)+2 {\rm Ind}(c)$ module $n$, it follows from \eqref{equation 2n+2/3leqind(c)leq} and Claim B that
$${\rm v}_c(T)-1\geq {\rm Ind}(c) > {\rm Ind}(3c)> \cdots>{\rm Ind}( (2\left\lfloor\frac{{\rm v}_c(T)-1}{2}\right\rfloor+1) c)\geq 1,$$ and in particular,
${\rm Ind}(c), {\rm Ind}(3c), \ldots,{\rm Ind}( (2\left\lfloor\frac{{\rm v}_c(T)-1}{2}\right\rfloor+1) c)$ is an arithmetic progression of positive integers with a difference $2 {\rm Ind}(c)-n\leq -2$.  Since the length of the above arithmetic progression is $\left\lfloor\frac{{\rm v}_c(T)-1}{2}\right\rfloor+1$,
it follows that the difference $$2 {\rm Ind}(c)-n=-2,$$ otherwise ${\rm Ind}(c)\geq {\rm Ind}( (2\left\lfloor\frac{{\rm v}_c(T)-1}{2}\right\rfloor+1) c)+3\left\lfloor\frac{{\rm v}_c(T)-1}{2}\right\rfloor\geq
1+3\left\lfloor\frac{{\rm v}_c(T)-1}{2}\right\rfloor>{\rm v}_c(T)-1$ which is absurd.
It follows that
$n \equiv 0\pmod 2$ and
\begin{equation}\label{equation Indc=frac n2-1}
{\rm Ind}(c) =\frac{n}{2}-1.
\end{equation}
By \eqref{equation (m-1)indleq n/2-1} with $m=2$, we derive that $\sum\limits_{a\mid Tc^{-{\rm v}_c(T)}} {\rm Ind}(a)=\frac{n}{2}$.
By Claim A, we conclude that $T=g^{\frac{n}{2}}\cdot c^{\frac{n}{2}}$.
If $4\not\mid n$, then $c^{\frac{n}{2}}$ is a minimal zero-sum subsequence of $T$ by \eqref{equation Indc=frac n2-1}, which is a contradiction with \eqref{equation V not in}. Hence, $$4\mid n
\mbox{ and } n\geq 8$$ by \eqref{equation ngeq 8}.
It follows from \eqref{equation Indc=frac n2-1} that $V=g^2 \cdot c^2=g^2 \cdot [(\frac{n}{2}-1)g]^2\in X$, a contradiction with \eqref{equation V not in}. This completes the proof of Theorem \ref{Theorem main result}. \qed

\bigskip

Next we shall prove Theorem \ref{Theorem intersection minimal} as follows.

\noindent {\bf Proof of Theorem \ref{Theorem intersection minimal}}. \
{\sl Arguments for Conclusion (i)}. \ For a sequence $V$ such that (i) $V=0$ or (ii) $V=a^{{\rm ord}(a)}$ with ${\rm ord}(a)=\frac{n}{2}$ or (iii) $V=a^{{\rm ord}(a)}$ with ${\rm ord}(a)=n$ or (iv) $V=g^{k}\cdot (-kg)$ with ${\rm ord}(g)=n, 1\leq k\leq n-2$, we construct a sequence $T_V$ over $G$ of length ${\rm D}(G)=n$ with (i) $T_V=0\cdot L$ where $L$ is zero-sum free sequence of length $n-1$ or (ii) $T_V=a^{n-1} b$ with ${\rm ord}(b)=n$ or (iii) $T_V=a^{{\rm ord}(a)}$ or (iv) $T_V=g^{n-1}\cdot (-kg)$, respectively. Observe that every zero-sum subsequence of $T_V$ is equal to $V$.  By Lemma \ref{lemma criterion} (ii), we conclude that $\{a^{{\rm ord}(a)}: a\in G, {\rm ord}(a)\in \{1,\frac{n}{2}, n\}\}\cup \{g^{k}\cdot (-kg):  g\in G, {\rm ord}(g)=n, 1\leq k\leq n-2\}\subset Q_n(G)$.
On the other hand, by Definitions \ref{Definition minimal set} and \ref{Definition Qk(G)}, we have \begin{equation}\label{equation some inclusions}
Q_n(G)=\bigcap\limits_{\Omega \subset \mathcal{B}(G), \ d_{\Omega}(G)=n}\Omega\subset \bigcap\Omega'
 \end{equation}
 where $\Omega'$ takes all minimal sets contained in $\mathcal A (G)$.  To complete the proof, we assume to the contrary that there exists a subsequence
\begin{equation}\label{equation Win setminus}
W\in Q_n(G)\setminus \{a^{{\rm ord}(a)}: a\in G, {\rm ord}(a)\in \{1,\frac{n}{2}, n\}\}\cup \{g^{k}\cdot (-kg):  g\in G, {\rm ord}(g)=n, 1\leq k\leq n-2\}.
\end{equation}
Combined with \eqref{equation some inclusions} and Theorem \ref{Theorem main result}, we have that  (i) $W=a^{{\rm ord}(a)}$ with $1<{\rm ord}(a)\leq \frac{n}{3}$ and $n\geq 6$; or (ii) $W=g^2\cdot [(\frac{n}{2}-1)g]^2$ with  ${\rm ord}(g)=n$ where $4\mid n$ and $n\geq 8$.
By Lemma \ref{lemma criterion} (ii), there exists a sequence $T$ over $G$ of length ${\rm D}(G)=n$ such that every zero-sum subsequence of $T$ is equal to $W$.

Consider the case of (i) $W=a^{{\rm ord}(a)}$ with ${\rm ord}(a)=\frac{n}{d}$ and $3\leq d<n$. Then $TW^{-1}$ is a zero-sum free sequence of length
\begin{equation}\label{equation |Tw-1|geq 2n/3}
|TW^{-1}|=|T|-{\rm ord}(a)=\frac{(d-1)n}{d}.
\end{equation}
Since $\frac{(d-1)n}{d}>\frac{n}{2}$, it follows from Lemma \ref{SachenChen} (i) that there exists some generator $b$ of $G$ such that
\begin{equation}\label{equation sum Indbaleq n-1}
\sum\limits_{a\mid T\cdot W^{-1}} {\rm Ind}_b(a)\leq n-1.
\end{equation}
Notice that
$\sum\limits_{a\mid T\cdot W^{-1}} {\rm Ind}_b(a)={\rm v}_b(T\cdot W^{-1})+\sum\limits_{a\mid T\cdot W^{-1}\cdot b^{-{\rm v}_b(TW^{-1})}} {\rm Ind}_b(a)\geq  {\rm v}_b(T\cdot W^{-1})+2(|T\cdot W^{-1}|-{\rm v}_b(T\cdot W^{-1}))=2|T\cdot W^{-1}|-{\rm v}_b(T\cdot W^{-1})$. Then it follows from \eqref{equation |Tw-1|geq 2n/3} and \eqref{equation sum Indbaleq n-1} that
\begin{equation}\label{equation vg(Tw-1)geq (k-2)n/k+1}
{\rm v}_b(T\cdot W^{-1})\geq  2|T\cdot W^{-1}|-\sum\limits_{a\mid T\cdot W^{-1}} {\rm Ind}_b(a)\geq \frac{(d-2)n}{d}+1\geq  \frac{(d-2)2d}{d}+1\geq d.
 \end{equation}
Since $\sum(W)=\langle a \rangle$ is the subgroup of $G$ of order $\frac{n}{d}$ and $\langle a \rangle=\{db,2db,\ldots,(n-d)b\}$, it follows that there exists some $m\in [1,{\rm ord}(a)]$ such that $ma=(n-d)b\in \langle a \rangle$. Combined with \eqref{equation vg(Tw-1)geq (k-2)n/k+1}, we find
a zero-sum subsequence $a^m b^d$ of $T$ which is not equal to $W$,  a contradiction.

Consider the case of  (ii) $W=g^2\cdot [(\frac{n}{2}-1)g]^2$ with  ${\rm ord}(g)=n$ where $4\mid n$ and $n\geq 8$.
Note that ${\rm v}_g(T)\leq 3$ or ${\rm v}_{(\frac{n}{2}-1)g}(T)\leq 3$, since otherwise $g^4 \cdot [(\frac{n}{2}-1)g]^4$ is a zero-sum subsequence of $T$ which is not equal to $W$, a contradiction.

Take an element $c\in \{g,(\frac{n}{2}-1)g\}$ such that ${\rm v}_{c}(T)\leq 3$, i.e.,  ${\rm v}_{c}(T)\in \{2,3\}$.
Then $T_1=T\cdot c^{-{\rm v}_{c}(T)+1}$ is a zero-sum free sequence of length $|T_1|=|T|-{\rm v}_{c}(T)+1\in \{n-1,n-2\}$. Since $|{\rm supp}(T_1)|\geq 2$, it follows from Lemma \ref{SachenChen} (i) that
$n-1\geq \sum\limits_{a\mid T_1} {\rm Ind}_f(a)\geq |T_1|+1\geq n-1$ for some generator $f$ of $G$, and so $\sum\limits_{a\mid T_1} {\rm Ind}_f(a)=n-1$ and $|T_1|=n-2$, which implies that $T_1=f^{n-3}(2f)$.
Since $n-3\geq 5$, we have that $c=2f$ and $T=f^{n-3} (2f)^{3}$. Then $f^{n-6} \cdot  (2f)^{3}$ is a zero-sum subsequence of $T$ which is not equal to $W$, a contradiction.  This completes the proof of Conclusion (i).

{\sl Arguments for Conclusion (ii)}. \ Let $t\in [n+1,2n-1]$. For any element $a\in G$ with ${\rm ord}(a)=n$, we can take a sequence $a^t\in \mathcal{F}(G)$ such that every nonempty zero-sum subsequence of $a^t$ is equal to $a^n$. By Lemma \ref{lemma criterion} (ii), we conclude that $$\{a^{n}: a\in G, {\rm ord}(a)=n\}\subset Q_{t}.$$
By Conclusion (i) and Lemma \ref{Lemma Qk}, it suffices to show that  $(\{a^{{\rm ord}(a)}: a\in G, {\rm ord}(a)\in \{1, \frac{n}{2}\}\}\cup \{g^{k}\cdot (-kg):  g\in G, {\rm ord}(g)=n, 1\leq k\leq n-2\})\bigcap Q_{n+1}=\emptyset$. Assume to the contrary that there exists some $$V\in \left(\{a^{{\rm ord}(a)}: a\in G, {\rm ord}(a)\in \{1, \frac{n}{2}\}\}\cup \{g^{k}\cdot (-kg):  g\in G, {\rm ord}(g)=n, 1\leq k\leq n-2\}\right)\bigcap Q_{n+1}.$$ It follows from Lemma \ref{lemma criterion} (ii) that there exists a sequence $T$ of length $n+1$ such that every nonempty zero-sum subsequence of $T$ is equal to $V$.

Suppose $V=0$. Then $|TV^{-1}|=n={\rm D}(G)$ and $TV^{-1}$ contains a nonempty zero-sum subsequence $W$. Then $V\cdot W$ is a nonempty zero-sum subsequence of $T$ which is not equal to $V$, a contradiction.

Suppose $V=a^{\frac{n}{2}}$ with ${\rm ord}(a)=\frac{n}{2}$. We have that ${\rm v}_a(TV^{-1})<\frac{n}{2}$, since otherwise $a^n$ is a zero-sum subsequence of $T$ which is not equal to $V$, a contradiction. Then $T\cdot V^{-1}$ contains two terms $b_1,b_2$ which are distinct from $a$.
Since $\langle a\rangle$ is the subgroup of $G$ with index $2$, it follows that at least one of the three elements $b_1,b_2, b_1+b_2$ belongs to $\langle a\rangle$. Say $b_1+b_2\in\langle a\rangle$. Since $\sum(V)=\langle a\rangle$, we can find a subsequence $V'$ of $V$ with $\sigma(V')=-(b_1+b_2)$ and thus, $b_1\cdot b_2\cdot V'$ is a zero-sum subsequence of $T$ which is not equal to $V$, a contradiction.

Suppose $V=g^{k}\cdot (-kg)$ with ${\rm ord}(g)=n$ and $1\leq k\leq n-2$. Combined with Lemma \ref{Lemma Graham conjecture}, we conclude that ${\rm supp}(T)=\{g,-kg\}$ and $T=g^{\ell}\cdot (-kg)^{n+1-\ell}$ with $k\leq \ell\leq n-1$.
Note that $k\neq \frac{n}{2},$ since otherwise $(-kg)^2$ is a zero-sum subsequence of $T$ with $(-kg)^2\neq V$, which is a contradiction.
Note that
\begin{equation}\label{equation -2kg3midT}
\ell<n-1,
\end{equation}
 since otherwise $\ell=n-1$ and
$2(-kg)\in\{g,2g,\ldots,(n-1)g\}=-\sum(g^{\ell})$ and then we find a zero-sum subsequence of $T$ within $(-kg)$ appearing twice which definitely is not equal to $V$, a contradiction.
Note that
\begin{equation}\label{equation -2kgneqg}
-2kg\neq g,
\end{equation} since otherwise $k=\frac{n-1}{2}$ and so $g^{k-1} \cdot (-kg)^3$ is a zero-sum subsequence of $T$ with $g^{k-1} (-kg)^3\neq V$, a contradiction. Then it follows from \eqref{equation -2kg3midT} and \eqref{equation -2kgneqg}
that $\widetilde{T}=g^{\ell}\cdot (-kg)^{n-1-\ell}\cdot (-2kg)$ is a sequence over $G$ of length $n$ with $|{\rm supp}(\widetilde{T})|=3$. By applying Lemma \ref{Lemma Graham conjecture}, we can find a nonempty zero-sum subsequence $W$ of $\widetilde{T}$ with $W\neq V$. Recall every nonempty zero-sum subsequence of $T$ is equal to $V=g^k \cdot (-kg)$, i.e., $W\not\mid T$. Then we conclude that $(-2kg)\mid W$ and so $W\cdot(-2kg)^{-1}\cdot (-kg)^2$ is a zero-sum subsequence of $T$ which is not equal to $V$, a contradiction.   This proves Conclusion (ii).

{\sl Arguments for Conclusion (iii)}. \ In fact, with Conclusion (iii), one can refer to  Theorem 4.4 of \cite{GaoLiPengWang} or Corollary 5.10 of
\cite{GaoHong1}. For the reader's convenience, we give a one-sentence argument here: Since every sequence over $G$ of length at least $2n$ contains two nonempty zero-sum subsequences with distinct lengths, then Conclusion (iii) follows readily by applying Lemma \ref{lemma criterion} (ii).
\qed

\section{The minimal set in general finite abelian groups}

In this section, we shall show the following theorem.

\begin{theorem}\label{theorem in group exp(G)geq 6} \ Let $G$ be a finite nonzero abelian group with $\exp(G)\neq 3$. Then $\mathcal A (G)$ is a minimal set with respect to the Davenport constant ${\rm D}(G)$ if and only if one of the following conditions holds: (i)  $G\cong \mathbb{Z}_4$; (ii) $G\cong \mathbb{Z}_5$; (iii) $G\cong \mathbb{Z}_2^r$ for some $r\geq 1$.
\end{theorem}

To prove Theorem \ref{theorem in group exp(G)geq 6}, some lemmas are necessary.

\begin{lemma} \cite{EmdeBoas007, Olson1}\label{Lemma Davenport precise value} \ Let $G\cong \mathbb{Z}_{n_1}\oplus \cdots  \oplus \mathbb{Z}_{n_r}$ with $1< n_1 \mid\cdots\mid n_r$. Then
${\rm D}(G)\geq 1 +\sum\limits_{i=1}^r (n_i-1)$. Moreover, equality holds if one of the following conditions holds:
(i) $r\leq 2$; (ii) $G$ is a $p$-group.
\end{lemma}

\begin{lemma} (see \cite{GRuzsa} Corollary 2.1.4) \label{Lemma 2group} \ Let $G\cong \mathbb{Z}_2^r$ with $r\geq 1$. A sequence $S\in \mathcal{F}(G)$ is zero-sum free if and only if $S$ is squarefree (i.e., $S$ contains no repeated terms) and ${\rm supp}(S)$ is an independent set.
\end{lemma}

Now we are in a position to prove Theorem \ref{theorem in group exp(G)geq 6}.

\noindent {\sl Proof of Theorem \ref{theorem in group exp(G)geq 6}}.
If
(i) or (ii) holds, then $\mathcal A (G)$ is a minimal set by Corollary \ref{Cor A(G) is not minimal} trivially. Suppose $G\cong \mathbb{Z}_2^r$ for some $r\geq 1$. Take any $V\in\mathcal A (G)$. To prove $\mathcal A (G)$ is a minimal set, by the arbitrariness of $V$, it suffices to show that $V$ belongs to every minimal set contained in $\mathcal A (G)$, i.e., there exists a sequence $T$ of length ${\rm D}(G)$ such that every minimal zero-sum subsequence of $T$ is equal to $V$ by Lemma \ref{lemma criterion} (i).

If $V=0$, the existence of $T$ follows by taking $T=V\cdot L$ where $L\in \mathcal{F}(G)$ is a zero-sum free sequence of length ${\rm D}(G)-1$. Hence, we may assume that $V\neq 0$ and so $|V|\geq 2$. Choose a term $a\mid V$. Say $V\cdot a^{-1}=b_1\cdot\ldots\cdot b_k$ ($k\geq 1$).
 Since $V\cdot a^{-1}$ is zero-sum free, it follows from Lemma \ref{Lemma 2group} that $\{b_1,\ldots,b_k\}$ is an independent set. Since $G$ is a vector space over the field $\mathbb{F}_2$, we see that $\{b_1,\ldots,b_k\}$ is contained in a basis, say $\{b_1,\ldots, b_r\}$, of the vector space $G$ of dimension $r$. Since $\sigma(V)=0$, it follows that $a=-\sum\limits_{i=1}^k b_i=\sum\limits_{i=1}^k b_i$. Since  $\{b_1,\ldots, b_r\}$ is a basis, it follows that $a=\sum\limits_{i=1}^k b_i$ is the unique way to represent the element $a$ as a linear combination of $\{b_1,\ldots, b_r\}$, and so $V$ is the unique nonempty zero-sum subsequence of the sequence $T=a\cdot b_1\cdot\ldots\cdot b_r$. By Lemma \ref{Lemma Davenport precise value}, $|T|=r+1={\rm D}(G)$.  This completes the argument for the sufficiency.

To prove the necessity, it suffices to show that when $\exp(G)\geq 4$ and $G\not\cong \mathbb{Z}_4$ and $G\not\cong \mathbb{Z}_5$, the set $\mathcal A (G)$ is not a minimal set with respect to ${\rm D}(G)$.

If $G$ is a cyclic group, then $|G|\geq 6$, then the conclusion follows from Corollary \ref{Cor A(G) is not minimal} . Hence, we assume that $G$ is not cyclic. Say $G\cong \mathbb{Z}_{n_1}\oplus \cdots  \oplus \mathbb{Z}_{n_r}$ with $r\geq 2$ and $1< n_1 \mid\cdots\mid n_r$. Then $n_r=\exp(G)\geq  4$. Let $\{e_1,\ldots,e_r\}$ be a basis of $G$ where ${\rm ord}(e_i)=n_i$ for $i\in [1,r]$.
Let $$V=(\prod\limits_{i\in [1,r-1]}e_i^{n_i-1})\cdot e_r^{n_r-3}\cdot (2e_r)\cdot (e_1+\cdots+e_r).$$
 Observe that $V$ is a minimal zero-sum sequence over $G$. By Lemma \ref{Lemma Davenport precise value}, we see that $|V|=\sum\limits_{i=1}^r (n_i-1)<{\rm D}(G)$. Moreover, we can verify that for any element $c\in G$, then $V\cdot c$ contains a minimal zero-sum subsequence which is not equal to $V$. By applying Lemma \ref{lemma criterion} (i), we conclude that
$V\notin \bigcap\limits_{\Omega \in \Gamma} \Omega$
where $\Gamma=\{\Omega\subset \mathcal{A}(G): \Omega \mbox{ is a minimal set} \}$. This implies that $\mathcal A (G)$ is not a minimal set with respect to ${\rm D}(G)$, completing the proof. \qed

Notice that Theorem \ref{theorem in group exp(G)geq 6} left us with the case of $\exp(G)=3$, i.e., $G\cong\mathbb{Z}_3^r$ with $r\geq 1$.  Then we close this section with the following conjecture.

\begin{conj} \  $\mathcal A (G)$ is a minimal set with respect to ${\rm D}(G)$ if $G\cong\mathbb{Z}_3^r$ for all $r\geq 1$.
\end{conj}

\section{Weighted Davenport constant for infinite abelian groups}

We start this section by the following definition of weighted Davenport constant for abelian groups.

\begin{definition} \cite{ZengYuanDiscrete} \ {\sl Let $F$ and $G$ be abelian groups. For any nonempty subset $\Psi$ of ${\rm Hom}(F, G)$, we define the $\Psi-$weighted Davenport constant of $G$, denoted ${\rm D}_{\Psi}(G)$, to be the least positive integer $\ell$ (if it exists, otherwise we let ${\rm D}_{\Psi}(G)=\infty$)
such that every sequence $T$ over $F$ of length
$\ell$ is not $\Psi$-zero-sum free,  i.e.,  there exists $a_1\cdot\ldots\cdot a_k\mid T$ ($k>0$),
such that $\sum\limits_{i=1}^{k}  \psi_i(a_i)=0_G$ for some $\psi_1,\ldots,\psi_{k}\in \Psi$.}
\end{definition}

Remark that Yuan and Zeng \cite{ZengYuanDiscrete}  introduce the above weighted Davenport constant for finite abelian groups aiming to prove the weighted generalization of the well-known zero-sum Gao's Theorem, and that their original definition  can be applied for infinite abelian groups as above. In order to investigate the weighted Davenport constant and some other zero-sum invariants in the realm of infinite abelian groups more systematically, we introduce the {\bf weighted universal zero-sum invariant}  as follows.

\begin{definition}\label{Definition Davenportinfiniteabelian}  \ {\sl Let $F$ and $G$ be abelian groups, and let $\Psi$ be a nonempty subset of ${\rm Hom}(F, G)$.
For any nonempty subset $\Omega \subset \mathcal F (G)$,  define ${\mathsf d}_{\Omega; \Psi}(G)$ to be the least positive integer $\ell$ (if it exists, otherwise we let ${\mathsf d}_{\Omega; \Psi}(G)=\infty$) such that
every sequence over $F$ of length $\ell$ has a $\Psi$-subsequence belonging to $\Omega$.
}
\end{definition}

We remark that the notion ${\mathsf d}_{\Omega; \Psi}(G)$ generalize the universal zero-sum invariant:
if $F=G$ and $\Psi$ is a singleton containing only the automorphism $1_G$, then  ${\mathsf d}_{\Omega; \Psi}(G)$ reduce to the invariant ${\mathsf d}_{\Omega}(G)$. On the other hand, for any abelian groups $F, G$, and any sets $\Psi\subset {\rm Hom}(F, G)$ and $\Omega\subset \mathcal{F}(G)$, the weighted universal invariant ${\mathsf d}_{\Omega; \Psi}(G)$ can be represented as an ordinary universal zero-sum invariant of the group $F$ as ${\mathsf d}_{\Omega; \Psi}(G)={\mathsf d}_{\Omega'}(F)$, where $\Omega'=\{T\in \mathcal{F}(F): \Psi(T)\cap \Omega\neq \emptyset\}$.
However, for the convenience of tacking problems we shall need this specified notation for weighted universal zero-sum invariant.

Now we show the following result on the existence of minimal sets associated with the  weighted universal zero-sum invariant, although such kind of questions usually turn to be plain in the setting of {\sl finite} abelian groups.

\begin{theorem}\label{Lemma result by Zorn's Lemma} Let $F$ and $G$ be abelian groups, and let $\Psi$ be a finite nonempty subset of ${\rm Hom}(F, G)$. For any $A\subset \mathcal{F}(G)$ such that ${\mathsf d}_{A; \Psi}(G)=t<\infty$, then there exists a set $\Omega\subset A$ which is minimal (with respect to set theoretic inclusion) with the property ${\mathsf d}_{\Omega; \Psi}(G)=t$.
\end{theorem}

\begin{proof} The proof is by applying Zorn's Lemma. Let $\mathcal{S}$ be the set of all subsets $C\subset A$ such that $d_{C;\Psi}(G)=t$. Since $A\in \mathcal{S}$, the set $\mathcal{S}$ is not empty.  Partially order $\mathcal{S}$ as: $C_1\prec C_2\Leftrightarrow C_1\supset C_2$ for any $C_1,C_2\in \mathcal{S}$. In order to apply Zorn's Lemma we must show that every chain $\{C_i: i\in I\}$ in $\mathcal{S}$ has an upper bound in $\mathcal{S}$. Let
\begin{equation}\label{equation E=cap Ci}
E=\bigcap\limits_{i\in I}C_i.
\end{equation}
To prove $E$ is an upper bound of the chain $\{C_i: i\in I\}$ in $\mathcal{S}$, it suffices to show that $E\in\mathcal{S}$, i.e.,
\begin{equation}\label{equation dE(G)=D(G)}
d_{E;\Psi}(G)=t.
\end{equation}
Since $E\subset C_i$, it follows from the definition that  $d_{E;\Psi}(G)\geq d_{C_i;\Psi}(G)=t$ for any $i\in I$.  To prove $d_{E;\Psi}(G)=t$, assume to the contrary that $d_{E;\Psi}(G)>t$. Then we find a sequence $T$ over $F$ of length $t$ such that there exists no $\Psi$-subsequence of $T$ belonging to $E$.
 Recall that both $|T|$ and $\Psi$ are finite. Let $\{V_1,\ldots,V_m\}$ be all nonempty $\Psi$-subsequences of $T$ where $m>0$. For each $i\in [1,m]$, since $V_i\notin E$, it follows that there exists some $k_i\in I$ such that $V_i\notin C_{k_i}$. Since $\{C_i: i\in I\}$ is a chain, it follows that there exists some $h\in [1,m]$ such that $C_{k_i} \prec C_{k_h}$, i.e., $C_{k_i}\supset C_{k_h}$, for each $i\in [1,m]$. Then $V_i\notin C_{k_h}$ for each $i\in [1,m]$, i.e., there exists no $\Psi$-subsequence of $T$ belonging to $C_{k_h}$, which implies that $d_{C_{k_h};\Psi}(G)\geq  |T|+1=t+1$, a contradiction.  This proves \eqref{equation dE(G)=D(G)}, and thus $E$ is an upper bound of the chain $\{C_i: i\in I\}$ in $\mathcal{S}$.
Thus the hypothesis of Zorn's Lemma are satisfied and hence $\mathcal{S}$ contains a maximal element. But a maximal element $\Omega$ of $\mathcal{S}$ is obviously the desired set in the theorem.
\end{proof}

We remark that if the condition $|\Psi|<\infty$ is removed in Theorem \ref{Lemma result by Zorn's Lemma}, then
the conclusion does not necessarily hold, which
can be seen as Example \ref{Exam in Z}.

Observed that ${\mathsf d}_{\Omega; \Psi}(G)=1$ if $\varepsilon\in \Omega$. So we also need only to consider the case that $\varepsilon\notin \Omega.$ As remarked before Definition  \ref{Definition Davenportinfiniteabelian},  from the definition we have
$${\rm D}_{\Psi}(G)={\mathsf d}_{\mathcal{B}(G)\setminus \{\varepsilon\}; \Psi}(G)={\mathsf d}_{\mathcal{A}(G); \Psi}(G).$$ In what follows, we shall use the notation ${\rm D}_{\Psi}(G)$ to denote the $\Psi$-weighted Davenport constant of $G$ for simplicity.
To investigate the weighted Davenport constant for infinite abelian groups, we have to answer the following basic question:

`{\sl When does ${\rm D}_{\Psi}(G)<\infty$ hold for given abelians groups $F, G$ and a nonempty subset $\Psi$ of ${\rm Hom}(F, G)$?}'

This question will be investigated in Theorem \ref{Prop finiteness} in terms of ${\rm Im} \ \psi$ and ${\rm Ker} \ \psi$ ($\psi\in \Psi$). To give a necessary and sufficient condition for ${\rm D}_{\Psi}(G)<\infty$, we shall need the following  result on covers of groups obtained by B.H. Neumann.

\begin{lemma}\label{lemma Neumann} \cite{Neumann1,Neumann2}  \ Let $\{a_iG_i\}_{i=1}^k$ be a finite cover of a group $G$ by
left cosets of subgroups $G_1,\ldots, G_k$. Then $G$ is the union
of those $a_iG_i$ with $[G : G_i] <\infty$.
\end{lemma}

\begin{lemma}\label{lemma sufficient condition} Let $F, G$ be abelian groups, and let $\Psi$ be a nonempty subset of  ${\rm Hom}(F, G)$. Then ${\rm D}_{\Psi}(G)<\infty$ provided that there exists some $\tau\in \Psi$  such  that  $|X\diagup {\rm Ker} \ \tau|<\infty$ where $X=F\setminus (\bigcup\limits_{\psi\in \Psi} {\rm Ker} \ \psi)$; furthermore,  ${\rm D}_{\Psi}(G)\leq |X\diagup {\rm Ker} \ \tau|+1$.
\end{lemma}

\begin{proof} Suppose that $|X\diagup {\rm Ker} \ \tau|<\infty$ for some $\tau\in \Psi$. Let \begin{equation}\label{equation ell=Xdiagup}
\ell=|X\diagup {\rm Ker} \ \tau|+1.
\end{equation}
 To show ${\rm D}_{\Psi}(G)\leq \ell$, we assume to the contrary that there exists a $\Psi$-zero-sum free sequence $T=a_1\cdot \ldots\cdot a_{\ell}$ over $F$ of length $\ell$. Then $\tau(a_1)\cdot\ldots\cdot \tau(a_{\ell})$ is zero-sum free. It follows that $\tau(a_1), \tau(a_1)+\tau(a_2),\ldots,\tau(a_1)+\cdots+\tau(a_{\ell})$ are $\ell$ pairwise distinct elements of ${\rm Im }\ \tau$, and so
$a_1, a_1+a_2,\ldots,a_1+\cdots+a_{\ell}$ are from distinct cosets of ${\rm Ker} \ \tau$. It follows from \eqref{equation ell=Xdiagup} that there exists some $t\in [1,\ell]$ such that $a_1+\cdots+a_{t}\notin X$, i.e., $a_1+\cdots+a_{t}\in {\rm Ker} \ \theta$ for some $\theta\in \Psi$.  Then $\theta(a_1)\cdot\ldots\cdot \theta(a_t)$ is a $\Psi$-zero-sum subsequence of $T$, a contradiction.
\end{proof}

\begin{theorem}\label{Prop finiteness} Let $F, G$ be abelian groups, and let $\Psi$ be a nonempty finite subset of  ${\rm Hom}(F, G)$. Then the following conditions are equivalent.

(i) ${\rm D}_{\Psi}(G)<\infty$;

(ii) There exists a finite cover of the group $F$ by left cosets of ${\rm Ker} \ \psi$ ($\psi\in \Psi$), i.e.,
$$F=\bigcup\limits_{\psi\in \Psi}\left(\bigcup\limits_{i=1}^{k_{\psi}}(a_{\psi, i}+ {\rm Ker}  \ \psi)\right),$$ where $k_{\psi}\in \mathbb{N}$ and $a_{\psi, 1}, \ldots, a_{\psi, k_{\psi}}\in F$ for each $\psi\in \Psi$;

(iii) $|{\rm Im} \ \tau|<\infty$ for some $\tau\in \Psi$.

\end{theorem}

\begin{proof} (i)$\Rightarrow$ (ii)  Let $T=b_1\cdot\ldots\cdot b_{\ell}$ be a sequence over $F$ of length $\ell={\rm D}_{\Psi}(G)-1$ such that $T$ is $\Psi$-zero-sum free. For each $i\in [1,\ell]$, we denote $B_i=\{\psi(b_i):\psi\in \Psi\}$.
Let $$S=(B_1\cup\{0_G\})+\cdots+(B_{\ell}\cup\{0_G\}).$$
Since $S$ is a finite subset of $G$, we have that $$\bigcup\limits_{\psi\in \Psi}\psi^{-1}(S)=\bigcup\limits_{\psi\in \Psi}\left(\bigcup\limits_{i=1}^{k_{\psi}}(a_{\psi, i}+ {\rm Ker}  \ \psi)\right)$$ where $k_{\psi}\in \mathbb{N}$ and $a_{\psi, 1}, \ldots, a_{\psi, k_{\psi}}\in F$. Then it suffices to show that $F=\bigcup\limits_{\psi\in \Psi}\psi^{-1}(S)$.
Take an arbitrary element $f$ of $F$. Since $|T\cdot f|={\rm D}_{\Psi}(G)$, it follows that $T\cdot f$ is not $\Psi$-zero-sum free any more, which implies that there exists $\tau\in\Psi$ ($\tau$ depends on $f$) such that $\tau(-f)=-\tau(f)\in S$, i.e., $-f\in \tau^{-1}(S)\subset \bigcup\limits_{\psi\in \Psi}\psi^{-1}(S)$.
By the arbitrariness of $f$, we have a finite cover $F=\bigcup\limits_{\psi\in \Psi}\psi^{-1}(S)=\bigcup\limits_{\psi\in \Psi}\left(\bigcup\limits_{i=1}^{k_{\psi}}(a_{\psi, i}+ {\rm Ker}  \ \psi)\right)$ by left cosets of these subgroups ${\rm Ker} \ \psi$ ($\psi\in \Psi$).

(ii)$\Rightarrow$ (iii) Since $\Psi$ is finite, we can obtain an irredundant cover of $F$ out of the cover $F=\bigcup\limits_{\psi\in \Psi}\left(\bigcup\limits_{i=1}^{k_{\psi}}(a_{\psi, i}+ {\rm Ker}  \ \psi)\right).$
By Lemma \ref{lemma Neumann}, we have that there exists some $\tau\in \Psi$ with some coset of ${\rm Ker}  \ \tau$ occurring in that irredundant cover such that $[F: {\rm Ker} \ \tau]<\infty$, and so $|{\rm Im} \ \tau|=[F: {\rm Ker} \ \tau]$ is finite, completing the proof.

(iii)$\Rightarrow$ (i) Note that $|X\diagup { {\rm Ker} \ \tau}|\leq |F\diagup { {\rm Ker} \ \tau}|=|{\rm Im} \ \tau|$. Then the conclusion follows by applying Lemma \ref{lemma sufficient condition}.
\end{proof}

We remark that in the case of $|\Psi|=\infty$,  the condition $|{\rm Im} \ \tau|<\infty$ for some $\tau\in \Psi$ (Theorem \ref{Prop finiteness} (iii)) is not necessary such that ${\rm D}_{\Psi}(G)<\infty$, which can be seen in Example \ref{Exam in Z} too.  In general, we give the following proposition to show this phenomenon.

\begin{prop}\label{theorem general finiteness} Let $t\in \mathbb{N}\cup \{+\infty\}$. There exist abelian groups $F$ and $G$ and a nonempty set $\Psi\subset {\rm Hom}(F, G)$, such that $${\rm D}_{\Psi}(G)=t$$ with $|{\rm Im} \ \psi|=\infty$ for each $\psi\in \Psi$.
\end{prop}

\begin{proof}
Let $F$ be an infinite abelian group
with the torsion subgroup $H$ such that $${\rm D}(H)=t$$ and the quotient group $F/H$ noncyclic.  Note that ${\rm D}(H)=t$ can be obtained by choosing $H$ to be a cyclic group of order $t$ or to be an infinite torsion group according to $t\in \mathbb{N}$ or $t=\infty$, respectively.
Take a subset $\{f_i\in F\setminus H: i\in I\}$ of $F\setminus H$ satisfying
\begin{equation}\label{equation G=bigcuplimiiin}
F=\left(\bigcup\limits_{i\in I}\langle f_i\rangle\right) \cup H.
\end{equation}
Observe that
\begin{equation}\label{equation cap=0}
\left(\bigcup\limits_{i\in I}\langle f_i\rangle\right)\cap H=\{0_F\}.
\end{equation}
 Let $G=\sum\limits_{i\in I} (F\diagup\langle f_i\rangle)$ be the direct sum of quotient groups $F\diagup\langle f_i\rangle$.
Let $\pi_i:F\rightarrow F\diagup\langle f_i\rangle$ be the canonical epimorphism, and let
\begin{equation}\label{equation Gdiagup}
\iota_i: F\diagup\langle f_i\rangle\rightarrow \sum\limits_{i\in I} (F\diagup\langle f_i\rangle)
\end{equation}
 be the canonical injection, and let
 \begin{equation}\label{equation psii=iotaipii}
 \psi_i=\iota_i\pi_i: F\rightarrow \sum\limits_{i\in I} (F\diagup\langle f_i\rangle),
  \end{equation}
  where $i\in I$. Observe that for each  $i\in I$,
\begin{equation}\label{equation Ker psii=gi}
{\rm Ker} \ \psi_i=\langle f_i\rangle,
\end{equation}
and  $$|{\rm Im} \ \psi_i|=\infty$$
 since $F\diagup H$ is noncyclic.
Let $\Psi=\{\psi_i: i\in I\}$.
It remains to prove ${\rm D}_{\Psi}(G)=t$.

Let $L$ be an arbitrary zero-sum free sequence of terms from the subgroup $H$.
By \eqref{equation cap=0} and \eqref{equation Ker psii=gi}, we derive that $\psi_i\mid_H$ is injective and so
$\psi_i(H)\cong H$ where $i\in I$. Combined with \eqref{equation Gdiagup} and \eqref{equation psii=iotaipii}, we conclude that $L$ contains no nonempty $\Psi$ zero-sum subsequence. This gives that ${\rm D}_{\Psi}(G)\geq  |L|$. By the arbitrariness of $L$, we have that ${\rm D}_{\Psi}(G)\geq t$ or ${\rm D}_{\Psi}(G)=\infty$ according to $t\in \mathbb{N}$ or $t$ takes infinity, respectively.

To completes the proof, it suffices to show ${\rm D}_{\Psi}(G)\leq t$ in case of $t\in \mathbb{N}$.
Let $T$ be a sequence over $F$ of length $t={\rm D}(H)$. If there exists some term $x\mid T$ such that $x\notin H$, it follows from \eqref{equation G=bigcuplimiiin} and \eqref{equation Ker psii=gi} that $\psi_i(x)=0_G$ for some $i\in I$,  done. Otherwise, all terms of $T$ are from $H$, it follows from $|T|={\rm D}(H)$ that $T$ contains a nonempty subsequence $V$ with $\sigma(V)=0_F$, and so $\prod\limits_{a\in V}\psi(a)$ is a $\Psi$-zero-sum subsequence of $T$ with any given $\psi\in \Psi$, done. \end{proof}

Note that in the above proposition, the constructed groups $F, G$ and $\Psi$ meet the sufficient condition $|X\diagup {\rm Ker} \ \tau|<\infty$ in Lemma \ref{lemma sufficient condition}. However, in Example \ref{Exam in Z}, we could expect ${\rm D}_{\Psi}(G)<\infty$ even though both $|{\rm Im}\ \tau|$ and $|X\diagup {\rm Ker} \ \tau|$ are infinite for each $\tau\in \Psi$. That is, we don't yet reach the possible necessary and sufficient conditions for ${\rm D}_{\Psi}(G)<\infty$ in case of $|\Psi|$ being {\sl infinite}. For this  we shall propose a problem in the end of this paper. Here we put Example \ref{Exam in Z} which has been used previously as illustration for a couple of times.

\begin{exam}\label{Exam in Z} Let $F=G=\mathbb{Z}$ be the additive group of integers. Let $\Psi=\{\psi_k: k\in \mathbb{Z}\setminus \{0\}\}$ consisting of
all nonzero (scalar) endomorphisms of the group $\mathbb{Z}$, where $\psi_k$ maps any integer $a$ to $ka$. Let $A=\mathcal{B}(G)\setminus \{\varepsilon\}$ consisting of all nonempty zero-sum sequences. We can verify that ${\rm D}_{\Psi}(G)={\mathsf d}_{A; \Psi}(G)=2$ and there exists no minimal set $\Omega$ contained in $A$ such that ${\mathsf d}_{\Omega; \Psi}(G)=2$.
\end{exam}

\noindent {\sl Sketch of the argument for Example \ref{Exam in Z}}.  The conclusion ${\mathsf d}_{A; \Psi}(G)=2$ holds clearly. To proceed, we suppose to the contrary that there exists a minimal subset $\Omega$ of $A$ such that ${\mathsf d}_{\Omega; \Psi}(G)=2$. By the minimality of $\Omega$, every sequence in $\Omega$ is of length at most $2$. Since $\Omega\neq \{0\}$, we can take a sequence of length $2$ out of $\Omega$, say $n\cdot (-n)\in \Omega$ for some $n>0$. By the minimality of $\Omega$ again, $(kn)\cdot (-kn)\notin \Omega$ for any $k>1$, since otherwise $d_{\Omega\setminus \{n\cdot (-n)\}}(G)={\rm D}_{\Psi}(G)$. Then we derive that there is no $\Psi$ subsequence of the sequence $(2n)\cdot (-2n)$ which belongs to $\Omega$. This implies ${\mathsf d}_{\Omega; \Psi}(G)>2$, a contradiction. \qed

\section{Concluding section}

Following from the celebrated Neumann Theorem on finite covers of groups, the problems on covers were investigated widely from the point of Combinatorics or of Algebras (see \cite{GaoGeroldingercover, Sun2,Tomkinson} e.g.). From Theorem  \ref{Prop finiteness}, we see that the finiteness of weighted Davenport constant is closely related to the finite cover of an infinite abelian group $G$ by cosets of given subgroups.
In 2003, Z.-W. Sun \cite{Sun} has established deep connections between investigations on zero-sum and on covers.
To motivate more work on the connections between both distinct topics,  we give the following two Propositions \ref{theorem translate} and \ref{prop valuesequality}.

\begin{prop}\label{theorem translate} Let $F$ be an abelian group, and let $H_1,\ldots,H_t$ be subgroups of $F$. Then the following two conditions are equivalent:

(i) There exists a finite cover of $F$ by the cosets of these groups $H_1,\ldots,H_t$;

(ii) There exists an abelian group $G$ and homomorphisms $\psi_1,\ldots,\psi_t\in {\rm Hom}(F, G)$ with ${\rm Ker} \ \psi_i=H_i$ ($i\in [1,t]$) such that ${\rm D}_{\Psi}(G)<\infty$, where $\Psi=\{\psi_1,\ldots,\psi_t\}$.
\end{prop}

\begin{proof} (ii) $\Rightarrow$ (i) follows from  Theorem \ref{Prop finiteness}  readily.

(i) $\Rightarrow$ (ii).  Let $F=\bigcup\limits_{i=1}^t \left(\bigcup\limits_{j=1}^{k_i}(a_{i, j}+H_i)\right)$ be a finite cover of $F$. It suffices to find the group $G$ and homomorphisms $\psi_1,\ldots,\psi_t$ with desired properties. Let $G=\sum\limits_{i=1}^{t} (F\diagup H_i)$ be the direct sum of quotient groups $F\diagup H_i$.
Let $\pi_i:F\rightarrow F\diagup H_i$ be the canonical epimorphism, and let
$\iota_i: F\diagup H_i \rightarrow \sum\limits_{i=1}^k (F\diagup H_i)$
 be the canonical injection, and let
$\psi_i=\iota_i\pi_i: F\rightarrow G$. Observe that $${\rm Ker} \ \psi_i=H_i.$$ By Lemma \ref{lemma Neumann},  we conclude that there exists $i\in [1,t]$ such that $[F: H_i]<\infty$, i.e., ${\rm Im} \ \psi_i\cong F\diagup H_i$ is finite. Then ${\rm D}_{\Psi}(G)<\infty$ follows from Theorem \ref{Prop finiteness}. This completes the proof of the theorem.
\end{proof}

Furthermore, with some result in quantitative investigation from cover theory for groups, in Proposition \ref{prop valuesequality} we can derive an upper bound for ${\rm D}_{\Psi}(G)$ by the size of irredundant finite covers, in case that $\Psi$ being finite.

\begin{lemma} \cite{Sun3} \label{Lemma Sun} Let $F$ be an abelian group and $\{a_i+H_i\}_{i=1}^k$ be a finite irredundant cover of $F$.
Then $[F:H_i]\leq 2^{k-1}$ for each subgroup $H_i$.
\end{lemma}

\begin{prop}\label{prop valuesequality} Let $F, G$ be abelian groups, and let $\Psi$ be a  finite nonempty subset of  ${\rm Hom}(F, G)$ such that ${\rm D}_{\Psi}(G)<\infty$.  Then ${\rm D}_{\Psi}(G)\leq 2^{M-1}$ where $M$ denotes the largest size of  irredundant finite covers of $F$ by cosets of subgroups ${\rm Ker} \ \psi$ ($\psi\in \Psi$).
\end{prop}

\begin{proof} By Proposition \ref{theorem translate}, we have a finite cover $\mathcal{S}$ of $F$ by cosets of subgroups ${\rm Ker} \ \psi$ ($\psi\in \Psi$). Take $\mathcal{S}'\subset \mathcal{S}$ to be an irredundant cover of $F$.  We see that $k=|\mathcal{S}'|\leq M$. Take a coset  $a+{\rm Ker} \ \tau \in \mathcal{S}'$. By Lemma \ref{Lemma Sun}, we have $[F: {\rm Ker} \ \tau]\leq 2^{k-1}$. It follows that ${\rm D}_{\Psi}(G)\leq {\rm D}_{\{\tau\}}(G)=D({\rm Im} \ \tau)\leq |{\rm Im} \ \tau|=[F: {\rm Ker} \ \psi]\leq 2^{k-1}\leq 2^{M-1}$, completing the proof.
\end{proof}

Then we close this paper with the following problem.

\begin{problem}\label{Problem 1} Let $F$ and $G$ be infinite abelian groups, and let $\Psi$ be an {\bf infinite} nonempty subset of  ${\rm Hom}(F, G)$. Find the necessary and sufficient conditions for ${\rm D}_{\Psi}(G) <\infty$ in terms of $\Psi$.
\end{problem}

\bigskip

\noindent {\bf Acknowledgements}

\noindent  This paper is presented to Professor Weidong Gao on his sixty and to appear on Communications in Algebra. 
This work is supported by NSFC (grant no. 12371335, 11971347, 12271520).


\begin{thebibliography}{99}

\bibitem{Adhikari1} S.D. Adhikari and Y.G. Chen, \emph{Davenport constant with weights and some related
questions. II,} J. Combin. Theory Ser. A, \textbf{115} (2008) 178--184.

\bibitem{Adhikari2} S.D. Adhikari, Y.G. Chen, J. B. Friedlander, S. V. Konyagin and F. Pappalardi,
\emph{Contributions to zero-sum problems,} Discrete Math., \textbf{306} (2006) 1--10.

\bibitem{Adhikari3} S.D. Adhikari, D.J. Grynkiewicz and Z.-W. Sun, \emph{On weighted zero-sum sequences,}
Adv. Appl. Math., \textbf{48} (2012) 506--527.


\bibitem{AGP} W.R. Alford, A. Granville and C. Pomerance, \emph{There are
infinitely many Carmichael numbers}, Ann. of Math., \textbf{140}
(1994) 703--722.



\bibitem{Boukheche} S. Boukheche, K. Merito, O. Ordaz and W.A. Schmid, \emph{Monoids of sequences over finite abelian groups defined via zero-sums with respect to a given set of weights
and applications to factorizations of norms of algebraic integers,} Comm. Algebra, \textbf{50} (2022) 4195--4217.

\bibitem{CziszterDoGerolding}  K. Cziszter, M. Domokos and A. Geroldinger, (2016) \emph{The Interplay of Invariant Theory with Multiplicative Ideal Theory and with Arithmetic Combinatorics. In: Chapman S., Fontana M., Geroldinger A., Olberding B. (eds) Multiplicative Ideal Theory and Factorization Theory,} Springer Proceedings in Mathematics $\&$ Statistics, vol 170. Springer, Cham.


\bibitem{EmdeBoas007} P. van Emde Boas, \emph{A combinatorial problem on finite abelian groups 2,} Report ZW-1969-007, Mathematical
Centre, Amsterdam, 1969.


\bibitem{EGZ} P. Erd\H{o}s, A. Ginzburg and A. Ziv, \emph{Theorem in additive number theory,}
Bull. Res. Council Israel, 10F (1961) 41--43.

\bibitem{GaoGeroldingercover} W. Gao and A. Geroldinger, \emph{Zero-sum problems and coverings by proper
cosets,} European J. Combin., \textbf{24} (2003) 531--549.


\bibitem{GaoGeroldingersurvey} W. Gao and A. Geroldinger, \emph{Zero-sum problems in finite abelian groups:
a survey,} Expo. Math.,  \textbf{24} (2006) 337--369.

\bibitem{GaoHamiWang}  W. Gao, Y. Hamidoune and G. Wang, \emph{Distinct length modular zero-sum subsequences: A proof of
Graham's conjecture,} J. Number Theory, \textbf{130} (2010) 1425--1431.

\bibitem{GaoHong1} W. Gao, S. Hong, W. Hui, X. Li, Q. Yin and P. Zhao, \emph{Representation of zero-sum invariants by sets of zero-sum
sequences over a finite abelian group,} Period. Math. Hungar., \textbf{85} (2022) 52--71.

\bibitem{GaoHong2} W. Gao, S. Hong, W. Hui, X. Li, Q. Yin and P. Zhao,
\emph{Representation of zero-sum invariants by sets of zero-sum sequences over a finite abelian group II,} J. Number Theory, \textbf{241} (2022) 738--760.

\bibitem{Index} W. Gao, Y. Li, J. Peng, C. Plyley and G. Wang, \emph{On the index of sequences over cyclic groups,}  Acta Arith., \textbf{148} (2011) 119--134.


\bibitem{GaoLiPengWang} W. Gao, Y. Li, J. Peng and G. Wang, \emph{A unifying look at zero-sum invariants,} Int. J. Number Theory, \textbf{14} (2018) 705--711.

\bibitem{GRuzsa} A. Geroldinger, \emph{Additive Group Theory and Non-unique Factorizations,} 1--86 in: A. Geroldinger and I. Ruzsa (Eds.), Combinatorial Number
Theory and Additive Group Theory (Advanced Courses in Mathematics-CRM Barcelona), Birkh\"{a}user, Basel, 2009.


\bibitem{GH} A. Geroldinger and F. Halter-Koch, \emph{Non-Unique
Factorizations. Algebraic, Combinatorial and Analytic Theory,}
Pure Appl. Math., vol. 278, Chapman $\&$ Hall/CRC,
2006.


\bibitem{GHZ} A. Geroldinger, F. Halter-Koch and Q. Zhong, \emph{ON MONOIDS OF WEIGHTED ZERO-SUM
SEQUENCES AND APPLICATIONS TO NORM
MONOIDS IN GALOIS NUMBER FIELDS AND
BINARY QUADRATIC FORMS}, Acta Math. Hungar., \textbf{168} (1) (2022) 144--185.


\bibitem{Grynkiewiczmono} D.J. Grynkiewicz, \emph{Structural Additive Theory}, Developments in Mathematics, vol. 30, Springer, Cham, 2013.

\bibitem{GrynkiewicMarchan}  D.J. Grynkiewicz, L.E. Marchan, and O. Ordaz, \emph{A weighted generalization of two
theorems of Gao,} Ramanujan J., \textbf{28} (2012) 323--340.

\bibitem{GrynkiewicIsrael} D.J. Grynkiewicz, A. Philipp and V. Ponomarenko, \emph{Arithmetic progression weighted
subsequence sums,} Israel J. Math., \textbf{193} (2013) 359--398.

\bibitem{LemkeKleitman} P. Lemke and D. Kleitman, \emph{An addition theorem on the integers modulo n,} J. Number Theory, \textbf{31} (1989) 335--345.

\bibitem{MarchanOrSaSc} L.E. Marchan, O. Ordaz, D. Ramos and W. A. Schmid, \emph{Some exact values of the
Harborth constant and its plus-minus weighted analogue,} Arch. Math., \textbf{101}
(2013) 501--512.


\bibitem{MaOrSaSc}  L.E. Marchan, O. Ordaz, I. Santos and W.A. Schmid, \emph{Multi-wise and constrained fully weighted Davenport constants and interactions,} J. Combin. Theory Ser. A, \textbf{135} (2015) 237--267.


\bibitem{Neumann1} B.H. Neumann, \emph{Groups covered by permutable subsets,} J. London Math. Soc.,
\textbf{29} (1954) 236--248.


\bibitem{Neumann2} B.H. Neumann, \emph{Groups covered by finitely many cosets,} Publ. Math. Debrecen,
\textbf{3} (1954) 227--242.

\bibitem{Olson1} J.E. Olson,  \emph{A Combinatorial Problem on Finite Abelian Groups, I}, J. Number Theory, \textbf{1} (1969) 8--10.


\bibitem{Rogers}  K. Rogers, \emph{A combinatorial problem in Abelian groups,} Math. Proc. Cambridge Philos.
Soc., \textbf{59} (1963) 559--562.



\bibitem{Savchev-Chen}     S. Savchev and F. Chen, \emph{Long zero-free sequences in finite cyclic groups,} Discrete Math., \textbf{307} (2007) 2671--2679.

\bibitem{Sun2}  Z.-W. Sun, \emph{Exact m-covers of Groups by Cosets,} European J. Combin., \textbf{22} (2001) 415--429.


\bibitem{Sun}  Z.-W. Sun, \emph{Unification of zero-sum problems, subset sums and covers of Z,}
Electron. Res. Announc. Am. Math. Soc., \textbf{9} (2003) 51--60.


\bibitem{Sun3}    Z.-W. Sun, \emph{On covers of abelian groups by cosets}, Acta Arith.,  \textbf{131} (2008) 341--350.



\bibitem{Tomkinson} M.J . Tomkinson \emph{Groups covered by finitely many cosets of subgroups,} Comm. Algebra,  \textbf{15} (1987) 845--859.

\bibitem{Yuan} P. Yuan, \emph{On the index of minimal zero-sum sequences over finite cyclic groups,} J. Combin. Theory Ser. A, \textbf{114} (2007)
1545--1551.



\bibitem{YuanZengEruo} P. Yuan and X. Zeng, \emph{Davenport constant with weights,} European J. Combin., \textbf{31} (2010) 677--680.

\bibitem{ZengYuanDiscrete}  X. Zeng and P. Yuan, \emph{Weighted Davenport's constant and the weighted EGZ Theorem,} Discrete Math., \textbf{311} (2011) 1940--1947.








\end{thebibliography}
\end{document}